\documentclass[12pt, a4paper]{article}

\usepackage{amsmath}
\usepackage{amsthm,amssymb,mathrsfs,amsmath}
\usepackage{graphicx}
\usepackage{float}
\usepackage{color}
\usepackage[top=1in, bottom=1in, left=1in, right=1in]{geometry}

\usepackage{subfigure}

\numberwithin{figure}{section}
\newtheorem{theorem}{Theorem}[subsection]
\newtheorem{lemma}{Lemma}[subsection]
\newtheorem{corollary}{Corollary}[subsection]
\newtheorem{proposition}{Proposition}[subsection]

\newtheorem{remark}{Remark}[subsection]
\newtheorem{definition}{Definition}[subsection]

\begin{document}

\title{Existence and nonexistence results of  radial solutions to singular BVPs arising in epitaxial growth theory}
\author{Amit Kumar Verma$^{a,\dagger}$, Biswajit Pandit$^b$\thanks{Email:$^a$akverma@iitp.ac.in,$^b$biswajitpandit82@gmail.com}
\\\small{\textit{$^{a,b}$Department of Mathematics, Indian Institute of Technology Patna, Patna--$801106$, Bihar, India.}}\\\small{\textit{$^\dagger$Dedicated to Late Prof. V. Lakshmikantham for his work on Monotone Iterative Technique. }}
}

\date{\today}

\maketitle

\begin{abstract}
The existence and nonexistence of stationary radial solutions to the elliptic partial differential equation arising in the molecular beam epitaxy are studied. The fourth-order radial equation is non-self adjoint and has no exact solutions. Also, it admits multiple solutions. Furthermore, solutions depend on the size of the parameter. We show that solutions exist for small positive values of this parameter. For large positive values of this parameter, we prove the nonexistence of solutions. We establish the qualitative properties of the solutions and provide bounds for the values of this parameter, which help us to separate the existence from nonexistence. We propose a new numerical scheme to capture the radial solutions. The results show that the iterative method is of better accuracy, more convenient and efficient for solving BVPs, which have multiple solutions. We verify theoretical results by numerical results. We also see the existence of solutions for negative values of the same parameter.
\end{abstract}

\section{Introduction}

 Epitaxy means the growth of a single thin film on top of a crystalline substrate. It is crucial for semiconductor thin-film technology, hard and soft coatings, protective coatings, optical coatings and etc. Epitaxial growth technique is used to produce the growth of semiconductor films and multilayer structures under high vacuum conditions (\cite{Barabasi:1995}). The major advantages of epitaxial growth are to reduces the growth time, better structural and superior electrical properties, eliminates the wastages caused during growth, wafering cost, cutting, polishing and etc. Several types of epitaxial growth techniques like the Hybrid vapor phase epitaxy (\cite{SO1997}), Chemical beam epitaxy (\cite{JSGJ1997}), Molecular beam epitaxy (MBE), etc have been used for the growth of compound semiconductors and other materials. In this work, we strictly focus on MBE, and we restrict our attention to the differential equation model, which is described by Carlos et. al. in \cite{Carlos2008,Carlos2012,CarlosOrigin2012,CarlosRadial2013}. In these references, the mathematical description of epitaxial growth is carried out by means of a function
\begin{equation}
\begin{aligned}\label{Eq 1}
\nonumber\sigma : \Omega \subset \mathbb{R}^{2}\times \mathbb{R}^{+} \rightarrow \mathbb{R},
\end{aligned}
\end{equation}
which describes the height of the growing interface in the spatial point $x \in \Omega \subset \mathbb{R}^{2} $ at time $t \in \mathbb{R}^{+}$. Authors (\cite{Carlos2008,Carlos2012,CarlosOrigin2012,CarlosRadial2013})  shown that the function $\sigma$   obeys the fourth order partial differential equation
\begin{equation}\label{P5Intro1}
\partial_t \sigma +\Delta^{2} \sigma~=~\text{det}(D^{2}\sigma)+\lambda \eta(x,t),~~~x\in \Omega \subset \mathbb{R}^{2},
\end{equation}
where $\eta(x,t)$ models the incoming mass entering the system through epitaxial deposition and $\lambda$ measures the intensity of this flux.
For simplicity they considered the stationary counterpart of the partial differential equation (\ref{P5Intro1}), which is given by 
\begin{equation}\label{P5Intro5}
\Delta^{2} \sigma~=~\text{det}(D^{2}\sigma)+\lambda G(x),~~~x\in \Omega \subset \mathbb{R}^{2},
\end{equation}
where they assumed that $\eta(x,t) \equiv G(x)$ is a stationary flux. Again, they set this problem on the unit disk and considered two types of boundary conditions. Corresponding to (\ref{P5Intro5}) homogeneous Dirichlet boundary condition (\cite{Carlos2012}) is
 \begin{equation}\label{P5Intro8}
 \displaystyle \sigma=0,~~~\frac{\partial \sigma}{\partial n}=0 ~~\mbox{on} ~~\partial \Omega,
 \end{equation}
where $n$ is unit out drawn normal to $\partial \Omega$, and homogeneous Navier boundary condition is 
\begin{equation}\label{P5Intro9}
\displaystyle \sigma=0,~~~\Delta \sigma=0~~\mbox{on} ~~\partial \Omega.
\end{equation} 
%\section{Reduction of the problem}
 By using the transformation  $r=|x|$ and $ \sigma(x)=\phi(|x|)$,  the above partial differential equation (\ref{P5Intro5}) is converted  into a fourth order ordinary differential equation which reads
\begin{equation}
  \begin{aligned}\label{P5Intro10}
&\frac{1}{r} \left \{ r\left[\frac{1}{r}(r \phi')^{'}\right]'\right \}^{'}=\frac{1}{r} \phi^{'} \phi^{''}+\lambda G(r),
 \end{aligned}
\end{equation}
where $\displaystyle '=\frac{d}{dr}$.

The  boundary conditions that correspond to (\ref{P5Intro8}) are 
\begin{eqnarray}\label{P5Eq205020}
\phi'(0)=0,~\phi(1)=0,~\phi'(1)=0,~\displaystyle \lim_{r\rightarrow0} r\phi'''(r)=0,
\end{eqnarray}
and the boundary conditions corresponding to (\ref{P5Intro9})
\begin{eqnarray}\label{P5Eq205021}
\phi'(0)=0,~\phi(1)=0,~\phi'(1)+\phi''(1)=0,~\displaystyle \lim_{r\rightarrow0} r\phi'''(r)=0.
\end{eqnarray}
Here, we impose another boundary conditions corresponding to (\ref{P5Intro9})
\begin{eqnarray}\label{P5Eq205022}
\phi'(0)=0,~\phi(1)=0,~\phi''(1)=0,~\displaystyle \lim_{r\rightarrow0} r\phi'''(r)=0.
\end{eqnarray}
The condition $\phi'(0)=0$  imposes the existence of an extremum at the origin. The conditions  $\phi(1)=0$ and $\phi'(1)=0$ are the actual boundary conditions. For simplicity we take  $G(r)=1$, which physically means that the new material is being deposited uniformly on the unit disc. Now, by using $\displaystyle \lim_{r\rightarrow0} r\phi'''(r)=0$, $ w=r \phi'$ and integrating by parts (\cite{Carlos2012}) from  equation (\ref{P5Intro10}), we have
\begin{equation}
\begin{aligned}\label{P5Eq1050}
&&&&r^{2} w''-r w'=\frac{1}{2} w^{2}+\frac{1}{2} \lambda r^{4}.
\end{aligned}
\end{equation}
By using the transformation $\displaystyle t=\frac{r^{2}}{2}$ and $u(t)=w(r)$, it is posible to reduce the equation (\ref{P5Eq1050}) into the following equation
 \begin{equation}\label{P1Eq1113}
 u''=\frac{u^{2}}{8t^{2}}+\frac{\lambda}{2},~\mbox{for}~t\in \left(0,\frac{1}{2}\right].
\end{equation}
Corresponding to (\ref{P1Eq1113}), we define the following three boundary value problems:
\begin{equation}
  \mbox{Problem 1:}~\left\{
    \begin{array}{@{} l c @{}}
      u''=\frac{u^{2}}{8t^{2}}+\frac{\lambda}{2},~\mbox{for}~t\in \left(0,\frac{1}{2}\right]  \\
      \lim_{t\rightarrow 0^{+}}\sqrt{t}u'(t)=0,~u(\frac{1}{2})=0,
    \end{array}\right.
  \label{P5problem1}
\end{equation} 
\begin{equation}
  \mbox{Problem 2:}~\left\{
    \begin{array}{@{} l c @{}}
      u''=\frac{u^{2}}{8t^{2}}+\frac{\lambda}{2},~\mbox{for}~t\in \left(0,\frac{1}{2}\right]  \\
      \lim_{t\rightarrow 0^{+}}\sqrt{t}u'(t)=0,~u'(\frac{1}{2})=0,
    \end{array}\right.
  \label{P5problem2}
\end{equation}
and 
\begin{equation}
  \mbox{Problem 3:}~\left\{
    \begin{array}{@{} l c @{}}
      u''=\frac{u^{2}}{8t^{2}}+\frac{\lambda}{2},~\mbox{for}~t\in \left(0,\frac{1}{2}\right]  \\
      \lim_{t\rightarrow 0^{+}}\sqrt{t}u'(t)=0,~u(\frac{1}{2})=u'(\frac{1}{2}).
    \end{array}\right.
  \label{P5problem3}
\end{equation}
 The BVPs (\ref{P5problem1}), (\ref{P5problem2}) and (\ref{P5problem3}) can equivalantly be described as the following integral equations (IE):
\begin{itemize}
\item IE corresponding to Problem $1$: 
\begin{equation}\label{P5problem1'}
u(t)=-\left[\left(\frac{1}{2}-t\right)\int_{0}^{t}\frac{u^{2}}{4s}ds+t\int_{t}^{\frac{1}{2}}\frac{u^{2}}{4s^{2}}\left(\frac{1}{2}-s\right)ds+\frac{\lambda}{4}t\left(\frac{1}{2}-t\right)\right],
\end{equation}
\item IE corresponding to Problem $2$:
\begin{equation}\label{P5problem2'}
u(t)=-\left[\int_{0}^{t}\frac{u^{2}}{8s}ds+t\int_{t}^{\frac{1}{2}}\frac{u^{2}}{8s^{2}}ds+\frac{\lambda}{4}t\left(1-t\right)\right],
\end{equation}
\item IE corresponding to Problem $3$:  
\begin{equation}\label{P5problem3'}
u(t)=-\left[\left(t+\frac{1}{2}\right)\int_{0}^{t}\frac{u^{2}}{4s}ds+t\int_{t}^{\frac{1}{2}}\frac{u^{2}}{4s^{2}}\left(s+\frac{1}{2}\right)ds+\frac{\lambda}{4}t\left(\frac{3}{2}-t\right)\right].
\end{equation}
\end{itemize}
We consider the function $u \in C^{2}_{loc}\left(\left(0,\frac{1}{2}\right];\mathbb{R}\right)$, where $C^{2}_{loc}\left(\left(0,\frac{1}{2}\right];\mathbb{R}\right)$ is defined as
$$C^{2}_{loc}\left(\left(0,\frac{1}{2}\right];\mathbb{R}\right)=\left\lbrace u:\left(0,\frac{1}{2}\right]\rightarrow \mathbb{R}|u\in C^{2}\left([a,b],\mathbb{R}\right) ~\mbox{for every compact set}~ [a,b]\subset \left(0,\frac{1}{2}\right]\right\rbrace.$$

In \cite{Carlos2012}, Carlos et. al. proved the existence and nonexistence of solutions of Problem $1$ and Problem $3$ by using upper and lower solution techniques. Corresponding to Problem $1$ and $3$, they have also provided the rigorous bounds of the values of the parameter $\lambda$, which helps us to separate the existence from nonexistence. We did not find any theoretical results corresponding to Problem $2$. Again, equation (\ref{P5Intro10}) is a nonlinear, singular, non-self-adjoint and has no exact solutions. Moreover, it admits multiple solutions. Therefore discrete methods such as finite element method etc may not be applicable to capture all solutions together. These facts highlight the difficulties to deal with such BVPs both analytically and numerically. Furthermore, to the best of our knowledge, there are only a few research papers that address both theoretical and numerical results corresponding to BVPs (\ref{P5problem1}), (\ref{P5problem2}) and (\ref{P5problem3}), and a lot of investigations are still pending. 

In this work, basically, we extend the theoretical results, which is described by Carlos et. al. in \cite{Carlos2012}. We prove some qualitative properties of the solutions and provide the rigorous bounds of the same parameter corresponding to different problems. To prove the existence of solutions, here we use the monotone iterative technique (\cite{PandeySBVP2008, Pandey2008, pandey2008existence, Pandey2010, PANDEY2010AMC, Cabada2007, Cabada2012, VL1985}). Recently, many researchers applied this technique on the initial value problem (IVP) for the nonlinear noninstantaneous impulsive differential equation (NIDE) (\cite{Ravi2017}), p-Laplacian boundary value problems with the right-handed Riemann-Liouville fractional derivative (\cite{Xue2019}), etc to prove the existence of the solution. Here, we also present numerical results to verify the theoretical results. We propose an iterative scheme to compute the approximate numerical solutions of the fourth-order differential equation (\ref{P5Intro10}) with $G(r) \equiv 1$ by using equations (\ref{P5problem1}), (\ref{P5problem2}), (\ref{P5problem3}) and it's respective Green's function. Recently, many authors have used numerical approximate methods like the Adomian decomposition method (ADM), homotopy perturbation method (HPM), etc to find approximate solution for different models involving differential equations (\cite{mital2008, odibat2007}), integral equation (\cite{KMMH1999, EBAD2005, SA2006}), fractional differential equations (\cite{lio2008, momani2008}) etc. After that, Waleed Al Hayani (\cite{Hayani2011}) and Singh et. al. (\cite{Singh2014}) applied ADM with Green's function to compute the approximate solution. They focused on the BVPs which have a unique solution. The major advantage of our proposed technique is to capture multiple solutions together with desired accuracy. 

The remainder of the paper has been focused on both theoretical and numerical results. We have proved some basic properties of the BVPs in section \ref{P5Basic}. The monotone iterative technique is presented in section \ref{P5Existence}, to prove the existence of a solution. A wide range of $\lambda$ of equation (\ref{P5Intro10}) corresponding to different types of boundary conditions are shown in section \ref{P5Estimation}. In section~\ref{P5ADM}, we apply our proposed technique on the integral equations and show a wide range of numerical results. Finally in section \ref{P5Conclusions}, we draw our main conclusions.
\section{Preliminary}\label{P5Basic} 
Corresponding to $\lambda\geq 0$, we prove some basic qualitative properties of the solution $u\in C^{2}_{loc}\left((0,\frac{1}{2}],\mathbb{R}\right)$, which satisfies the following inequality
\begin{equation}\label{P1Eq9}
u''\geq\frac{u^{2}}{8t^{2}}+\frac{\lambda}{2},~\mbox{for}~t\in \left(0,\frac{1}{2}\right].
\end{equation}
Here, we omit the proof of lemma \ref{P1Lemma1},  lemma \ref{P1Lemma2}, lemma \ref{P1Lemma3}, corollary \ref{P1Corollary1}, lemma \ref{P5lemma50} which has been done by Carlos et. al. in \cite{Carlos2012}.
\begin{lemma}\label{P1Lemma1}
Let $u\in C^{2}_{loc}\left((0,\frac{1}{2}],\mathbb{R}\right)$ satisfy $\displaystyle\lim_{t\rightarrow 0^{+}}\sqrt{t}u'(t)=0$ and equation (\ref{P1Eq9}), then $\displaystyle\lim_{t\rightarrow 0}u(t)=0$.
\end{lemma}
\begin{lemma}\label{P1Lemma2}
Let $u\in C^{2}_{loc}\left((0,\frac{1}{2}],\mathbb{R}\right)$ satisfy $\displaystyle\lim_{t\rightarrow 0}u(t)=0$, $u\left(\frac{1}{2}\right)=0$ and equation (\ref{P1Eq9}), then $u(t)\leq0$ for all $t\in \left(0,\frac{1}{2}\right]$.
\end{lemma}
\begin{lemma}\label{P1Lemma3}
Let $u\in C^{2}_{loc}\left((0,\frac{1}{2}],\mathbb{R}\right)$ satisfy $\displaystyle\lim_{t\rightarrow 0}u(t)=0$, $u\left(\frac{1}{2}\right)=u'\left(\frac{1}{2}\right)$ and equation (\ref{P1Eq9}), then $u(t)\leq0$ for all $t\in \left(0,\frac{1}{2}\right]$.
\end{lemma}
\begin{corollary}\label{P1Corollary1}
Let $u\in C^{2}_{loc}\left((0,\frac{1}{2}],\mathbb{R}\right)$ satisfy $\displaystyle\lim_{t\rightarrow 0}u(t)=0$, $u(t)\leq 0$ and equation (\ref{P1Eq9}), then $\displaystyle\lim_{t\rightarrow 0^{+}}\sqrt{t}u'(t)=0$ if and only if $\displaystyle\lim_{t\rightarrow 0^{+}}\frac{u(t)}{\sqrt{t}}=0$.
\end{corollary}
\begin{lemma}\label{P5lemma50}
Let $u\in C^{2}_{loc}\left((0,\frac{1}{2}],\mathbb{R}\right)$ satisfy $\displaystyle\lim_{t\rightarrow 0^{+}}\frac{u(t)}{\sqrt{t}}=0$. Then for every $\mu \in [0,1)$, we have 
\begin{equation}
\lim_{t \rightarrow 0^{+}} t^{1-\mu}\int_{t}^{\frac{1}{2}}\frac{u^{2}}{s^{2}} ds=0.
\end{equation} 
\end{lemma}
\begin{lemma}\label{P1Lemma4}
Let $u\in C^{2}_{loc}\left((0,\frac{1}{2}],\mathbb{R}\right)$ satisfy $\displaystyle\lim_{t\rightarrow 0}u(t)=0$, $u'\left(\frac{1}{2}\right)=0$ and equation (\ref{P1Eq9}), then $u(t)\leq0$ for all $t\in \left(0,\frac{1}{2}\right]$.
\end{lemma}
\begin{proof}
First, we show  that $u\left(\frac{1}{2}\right)\leq 0$. Assume $u\left(\frac{1}{2}\right)> 0$. Since $\displaystyle\lim_{t\rightarrow 0}u(t)=0$, therefore we have there exist a $t_{0}\in \left(0,\frac{1}{2}\right]$ such that $u(t_{0})<u\left(\frac{1}{2}\right)$. Now from (\ref{P1Eq9}), we have $u'(t)$ is increasing function on  $\left(0,\frac{1}{2}\right]$. Again by mean value theorem, we have 
\begin{equation}
\frac{u\left(\frac{1}{2}\right)-u(t_{0})}{\frac{1}{2}-t_{0}}=u'(\xi),~\min\left\lbrace\frac{1}{2},t_{0}\right\rbrace\leq \xi \leq \max\left\lbrace\frac{1}{2},t_{0}\right\rbrace.
\end{equation}
Since $u'\left(\frac{1}{2}\right)=0$, therefore we have $\displaystyle \left(u\left(\frac{1}{2}\right)-u(t_{0})\right)\leq u'\left(\frac{1}{2}\right)\left(\frac{1}{2}-t_{0}\right)= 0.$ Hence we get $u\left(\frac{1}{2}\right)\leq u(t_{0})$, which is a contradiction. So, we have $u\left(\frac{1}{2}\right)\leq 0.$ Furthermore, $u(t)$ is a convex function along with $u'\left(\frac{1}{2}\right)=0$. Also $u'(t)$ is increasing, which implies $u'(t)\leq 0$. Again $u(t)$ is decreasing function on $\left(0,\frac{1}{2}\right]$. Therefore $\displaystyle\lim_{t\rightarrow 0}u(t)=0$ and $u\left(\frac{1}{2}\right)\leq 0$ leads to $u(t)\leq 0$ on $ \left(0,\frac{1}{2}\right]$.
\end{proof}

\begin{lemma}\label{P1lemma5}
Let $u\in C^{2}_{loc}\left((0,\frac{1}{2}],\mathbb{R}\right)$ be the solution of Problem $3$, then $u(t)$ satisfies the following integral equation
\begin{equation}\label{P1Eq11}
u(t)=-\left[\left(t+\frac{1}{2}\right)\int_{0}^{t}\frac{u^{2}}{4s}ds+t\int_{t}^{\frac{1}{2}}\frac{u^{2}}{4s^{2}}\left(s+\frac{1}{2}\right)ds+\frac{\lambda}{4}t\left(\frac{3}{2}-t\right)\right],
\end{equation}
and 
\begin{equation}\label{P1Eq12}
\lim_{t\rightarrow 0^{+}}\frac{|u(t)|}{t}<+\infty.
\end{equation}
\end{lemma}
\begin{proof}
The Green's function of the Problem $3$ can be written as 
\begin{equation}
  G(s,t)=\left\{
    \begin{array}{@{} l c @{}}
      -2s\left(t+\frac{1}{2}\right)&0\leq s \leq t,  \\
      -2t\left(s+\frac{1}{2}\right)&t\leq s \leq \frac{1}{2}.
    \end{array}\right.
  \label{P1Eq13}
\end{equation}
Therefore from equation (\ref{P1Eq13}) and Problem $3$, we can easily deduce the integral equation (\ref{P1Eq11}). Now, by using the result of Lemma \ref{P1Lemma1}, we have
\begin{equation}\label{P1Eq1000}
\lim_{t\rightarrow0^{+}}\frac{|u(t)|}{t}=\left|\int_{0}^{\frac{1}{2}}\frac{u^{2}}{4s^{2}}\left(s+\frac{1}{2}\right)ds+\frac{3\lambda}{8}\right|.
\end{equation}
Now, put 
\begin{equation}
f(t)=\frac{u^{2}}{t},~g(t)=\frac{1}{t^\mu}~\mbox{and}~h(t)=\frac{1}{t^{1-\mu}}~~\mbox{for}~t\in\left(0,\frac{1}{2}\right].
\end{equation}
Therefore we get $fg\in L\left((0,\frac{1}{2}]\right)$ provided $\mu\in(0,1)$.  Consequently, we have 
\begin{equation}
\int_{0}^{\frac{1}{2}}\frac{u^{2}}{s^{2}}ds<\infty.
\end{equation}
Hence, from equation (\ref{P1Eq1000}), we get  the equation (\ref{P1Eq12}).
\end{proof}
\begin{lemma}\label{P1lemma6}
Let $u(t)\in C^{2}_{loc}\left((0,\frac{1}{2}],\mathbb{R}\right)$ be the solution of Problem $2$, then $u(t)$ can be written as in the following form
\begin{equation}\label{P51001}
u(t)=-\left[\int_{0}^{t}\frac{u^{2}}{8s}ds+t\int_{t}^{\frac{1}{2}}\frac{u^{2}}{8s^{2}}ds+\frac{\lambda}{4}t\left(1-t\right)\right],
\end{equation}
and also satisfies
\begin{equation}\label{P51002}
\lim_{t\rightarrow 0^{+}}\frac{|u(t)|}{t}<\infty.
\end{equation}
\end{lemma}
\begin{proof}
By using the boundary condition and properties of  Green's function, we have
\begin{equation}
  G(s,t)=\left\{
    \begin{array}{@{} l c @{}}
      -s&0\leq s \leq t,  \\
      -t&t\leq s \leq \frac{1}{2}.
    \end{array}\right.
  \label{P1Eq17}
\end{equation}
Similarly, from equation (\ref{P1Eq17}) and Problem $2$, we can easily derive the equation (\ref{P51001}). Now, by using the result of Lemma \ref{P1Lemma1}, we have
\begin{equation}\label{P5Eq1000}
\lim_{t\rightarrow0^{+}}\frac{|u(t)|}{t}=\left|\int_{0}^{\frac{1}{2}}\frac{u^{2}}{8s^{2}}ds+\frac{\lambda}{4}\right|.
\end{equation}
 Therefore, from equations (\ref{P5Eq1000}) and by similar analysis as in Lemma \ref{P1lemma5}, we can prove the result (\ref{P51002}).
\end{proof}
\begin{lemma}\label{P1lemma7}
Let $u(t)\in C^{2}_{loc}\left((0,\frac{1}{2}],\mathbb{R}\right)$ be the solution of Problem $1$, then $u(t)$ can be written as in the following form
\begin{equation}\label{P5Eq11000}
u(t)=-\left[\left(\frac{1}{2}-t\right)\int_{0}^{t}\frac{u^{2}}{4s}ds+t\int_{t}^{\frac{1}{2}}\frac{u^{2}}{4s^{2}}\left(\frac{1}{2}-s\right)ds+\frac{\lambda}{4}t\left(\frac{1}{2}-t\right)\right],
\end{equation}
and satisfies
\begin{equation}\label{P5Eq12000}
\lim_{t\rightarrow 0^{+}}\frac{|u(t)|}{t}<\infty.
\end{equation}
\end{lemma}
\begin{proof}
The Green's function of the Problem $1$  is given by
\begin{equation}
  G(s,t)=\left\{
    \begin{array}{@{} l c @{}}
      -2s\left(\frac{1}{2}-t\right)&0\leq s \leq t,  \\
      -2t\left(\frac{1}{2}-s\right)&t\leq s \leq \frac{1}{2}.
    \end{array}\right.
  \label{P1Eq20}
\end{equation}
Again, from equation (\ref{P1Eq20}) and Problem $3$, we derive integral equation (\ref{P5Eq11000}). Furthermore, by using the result of Lemma \ref{P1Lemma1}, we have
\begin{equation}\label{P5Eq10000}
\lim_{t\rightarrow0^{+}}\frac{|u(t)|}{t}=\left|\int_{0}^{\frac{1}{2}}\frac{u^{2}}{4s^{2}}\left(\frac{1}{2}-s\right)ds+\frac{\lambda}{8}\right|.
\end{equation}
Again, by similar analysis as in Lemma \ref{P1lemma6}, we get the inequality (\ref{P5Eq12000}).
\end{proof}
\begin{remark}
From Lemma \ref{P1lemma7} (respectively Lemma \ref{P1lemma6} and Lemma \ref{P1lemma5}), we can easily justify the result of Lemma \ref{P1Lemma2} (respectively Lemma \ref{P1Lemma4} and Lemma \ref{P1Lemma3}). 
\end{remark}
\section{Existence of solutions}\label{P5Existence}
In this section, we apply the monotone lower and upper solution technique to prove the existence of at least one solution of Problem $1$, Problem $2$ and Problem $3$. For this purpose, we need to prove some lemmas, which help us to proof the main theorems.
\subsection{Construction of Green's function}
To investigate the Problem $1$, Problem $2$ and Problem $3$, we consider the corresponding nonlinear singular boundary value problems, which are  given by
\begin{equation}
  \mbox{Problem $1(a)$:}~\left\{
    \begin{array}{@{} l c @{}}
      u''+ku=h(t),~\mbox{for}~t\in \left(0,\frac{1}{2}\right],  \\
      \lim_{t\rightarrow 0^{+}}\sqrt{t}u'(t)=0,~u\left(\frac{1}{2}\right)=0,
    \end{array}\right.
  \label{P1problem1}
\end{equation} 
\begin{equation}
  \mbox{Problem $2(a)$:}~\left\{
    \begin{array}{@{} l c @{}}
      u''+ku=h(t),~\mbox{for}~t\in \left(0,\frac{1}{2}\right],  \\
      \lim_{t\rightarrow 0^{+}}\sqrt{t}u'(t)=0,~u'\left(\frac{1}{2}\right)=0,
    \end{array}\right.
  \label{P1problem2}
\end{equation}
and 
\begin{equation}
  \mbox{Problem $3(a)$:}~\left\{
    \begin{array}{@{} l c @{}}
      u''+ku=h(t),~\mbox{for}~t\in \left(0,\frac{1}{2}\right],  \\
      \lim_{t\rightarrow 0^{+}}\sqrt{t}u'(t)=0,~u\left(\frac{1}{2}\right)=u'\left(\frac{1}{2}\right),
    \end{array}\right.
  \label{P1problem3}
\end{equation}
where $\displaystyle h(t)=\frac{u^{2}}{8t^{2}}+\frac{\lambda}{2}+ku$, $k\in \mathbb{R}$ and $\lambda\in \mathbb{R}$.

\begin{lemma}\label{P1Lemma8}
Let $k<0$ and $u(t)\in C^{2}_{loc}\left((0,\frac{1}{2}],\mathbb{R}\right)$ be the solution of Problem $1(a)$, then 
\begin{equation}\label{P10002}
u(t)=\int_{0}^{\frac{1}{2}}G(s,t)h(s)ds,
\end{equation}
where Green's function $G(s,t)$ is given by
\begin{equation}
  G(s,t)=\left\{
    \begin{array}{@{} l c @{}}
      -\frac{\sinh\left(\frac{\sqrt{|k|}}{2}-\sqrt{|k|}t\right)\sinh\left(\sqrt{|k|}s\right)}{\sqrt{|k|}\sinh\left(\frac{\sqrt{|k|}}{2}\right)}&0\leq s \leq t,  \\
      -\frac{\sinh\left(\frac{\sqrt{|k|}}{2}-\sqrt{|k|}s\right)\sinh\left(\sqrt{|k|}t\right)}{\sqrt{|k|}\sinh\left(\frac{\sqrt{|k|}}{2}\right)}&t\leq s \leq \frac{1}{2},
    \end{array}\right.
  \label{P1Eq25}
\end{equation}
and $G(s,t)\leq 0$ for all $0\leq s\leq \frac{1}{2}$ and $0\leq  t\leq \frac{1}{2}$.
\end{lemma}
\begin{proof}
By using the boundary condition of Problem $1(a)$ and properties of Green's function, we can easily prove the equation (\ref{P1Eq25}). Furthermore we have $G(s,t)\leq 0$ for all $0\leq s\leq \frac{1}{2}$ and $0\leq  t\leq \frac{1}{2}$.
\end{proof}
\begin{lemma}\label{P1Lemma9}
Let $k<0$ and $u(t)\in C^{2}_{loc}\left((0,\frac{1}{2}],\mathbb{R}\right)$ be the solution of Problem $2(a)$, then 
\begin{equation}
u(t)=\int_{0}^{\frac{1}{2}}G(s,t)h(s)ds,
\end{equation}
where Green's function $G(s,t)$ is given by
\begin{equation}
  G(s,t)=\left\{
    \begin{array}{@{} l c @{}}
      -\frac{\cosh\left(\frac{\sqrt{|k|}}{2}-\sqrt{|k|}t\right)\sinh\left(\sqrt{|k|}s\right)}{\sqrt{|k|}\cosh\left(\frac{\sqrt{|k|}}{2}\right)}&0\leq s \leq t,  \\
      -\frac{\cosh\left(\frac{\sqrt{|k|}}{2}-\sqrt{|k|}s\right)\sinh\left(\sqrt{|k|}t\right)}{\sqrt{|k|}\cosh\left(\frac{\sqrt{|k|}}{2}\right)}&t\leq s \leq \frac{1}{2},
    \end{array}\right.
  \label{P1Eq27}
\end{equation}
and $G(s,t)\leq 0$ for all $0\leq s\leq \frac{1}{2}$ and $0\leq  t\leq \frac{1}{2}$.
\end{lemma}
\begin{proof}
In a similar manner as in Lemma \ref{P1Lemma8}, we can easily get the equation (\ref{P1Eq27}), and prove $G(s,t)\leq 0$ for all $0\leq s\leq \frac{1}{2}$ and $0\leq  t\leq \frac{1}{2}$.
\end{proof}
\begin{lemma}\label{P1Lemma10}
Let $k<0$,
$\sqrt{|k|}\cosh\left(\frac{\sqrt{|k|}}{2}\right)-\sinh\left(\frac{\sqrt{|k|}}{2}\right)>0$ 
and $u(t)\in C^{2}_{loc}\left((0,\frac{1}{2}],\mathbb{R}\right)$ be the solution of Problem $3(a)$, then 
\begin{equation}
u(t)=\int_{0}^{\frac{1}{2}}G(s,t)h(s)ds,
\end{equation}
where Green's function $G(s,t)$ is given by
\begin{equation}
  G(s,t)=\left\{
    \begin{array}{@{} l c @{}}
      -\frac{\left[\sqrt{|k|}\cosh\left(\frac{\sqrt{|k|}}{2}-\sqrt{|k|}t\right)-\sinh\left(\frac{\sqrt{|k|}}{2}-\sqrt{|k|}t\right)\right]  \sinh\left(\sqrt{|k|}s\right)}{\sqrt{|k|}\left[\sqrt{|k|}\cosh\left(\frac{\sqrt{|k|}}{2}\right)-\sinh\left(\frac{\sqrt{|k|}}{2}\right)\right]}&0\leq s \leq t,  \\
      -\frac{\left[\sqrt{|k|}\cosh\left(\frac{\sqrt{|k|}}{2}-\sqrt{|k|}s\right)-\sinh\left(\frac{\sqrt{|k|}}{2}-\sqrt{|k|}s\right)\right]  \sinh\left(\sqrt{|k|}t\right)}{\sqrt{|k|}\left[\sqrt{|k|}\cosh\left(\frac{\sqrt{|k|}}{2}\right)-\sinh\left(\frac{\sqrt{|k|}}{2}\right)\right]}&t\leq s \leq \frac{1}{2},
    \end{array}\right.
  \label{P1Eq29}
\end{equation}
and  $G(s,t)\leq 0$ for all $0\leq s\leq \frac{1}{2}$ and $0\leq  t\leq \frac{1}{2}$.
\end{lemma}
\begin{proof}
Again by similar analysis, we can easily derive the equation (\ref{P1Eq29}). Now,
\begin{eqnarray}
&&\sqrt{|k|}\cosh\left(\frac{\sqrt{|k|}}{2}-\sqrt{|k|}t\right)-\sinh\left(\frac{\sqrt{|k|}}{2}-\sqrt{|k|}t\right),\\
&&=\sqrt{|k|}\cosh\left(\frac{\sqrt{|k|}}{2}\right)\cosh\left(\sqrt{|k|}t\right)-\sqrt{|k|}\sinh\left(\frac{\sqrt{|k|}}{2}\right)\sinh\left(\sqrt{|k|}t\right)-\\
&&\hspace{2cm} \sinh\left(\frac{\sqrt{|k|}}{2}\right)\cosh\left(\sqrt{|k|}t\right)+\cosh\left(\frac{\sqrt{|k|}}{2}\right)\sinh\left(\sqrt{|k|}t\right),\\
&&= \left[\sqrt{|k|}\cosh\left(\frac{\sqrt{|k|}}{2}\right)-\sinh\left(\frac{\sqrt{|k|}}{2}\right)\right]\cosh\left(\sqrt{|k|}t\right)\\
&&\hspace{2cm}-\sqrt{|k|}\sinh\left(\frac{\sqrt{|k|}}{2}\right)\sinh\left(\sqrt{|k|}t\right)+\cosh\left(\frac{\sqrt{|k|}}{2}\right)\sinh\left(\sqrt{|k|}t\right),
\end{eqnarray}
\begin{eqnarray}
&&\geq\left(\sqrt{|k|}\cosh\left(\frac{\sqrt{|k|}}{2}\right)-\sinh\left(\frac{\sqrt{|k|}}{2}\right)\right)\left(\cosh\left(\sqrt{|k|}t\right)-\sinh\left(\sqrt{|k|}t\right)\right),\\
&&\geq 0,~\mbox{since} ~\tanh\left(\sqrt{|k|}t\right)\leq1~\mbox{for all}~t\in \left(0,\frac{1}{2}\right].
\end{eqnarray}
Hence from (\ref{P1Eq29}), we have $G(s,t)\leq 0$ for all $0\leq s\leq \frac{1}{2}$ and $0\leq  t\leq \frac{1}{2}$.
\end{proof}

\begin{lemma}\label{P1Lemma8000}
Let $0<k<4\pi^{2}$ and $u(t)\in C^{2}_{loc}\left((0,\frac{1}{2}],\mathbb{R}\right)$ be the solution of Problem $1(a)$, then 
\begin{equation}
u(t)=\int_{0}^{\frac{1}{2}}G(s,t)h(s)ds,
\end{equation}
where Green's function $G(s,t)$ is given by
\begin{equation}
  G(s,t)=\left\{
    \begin{array}{@{} l c @{}}
      -\frac{\sin\left(\frac{\sqrt{k}}{2}-\sqrt{k}t\right)\sin\left(\sqrt{k}s\right)}{\sqrt{k}\sin\left(\frac{\sqrt{k}}{2}\right)}&0\leq s \leq t,  \\
      -\frac{\sin\left(\frac{\sqrt{k}}{2}-\sqrt{k}s\right)\sin\left(\sqrt{k}t\right)}{\sqrt{k}\sin\left(\frac{\sqrt{k}}{2}\right)}&t\leq s \leq \frac{1}{2},
    \end{array}\right.
 % \label{P1Eq25}
\end{equation}
and $G(s,t)\leq 0$ for all $0\leq s\leq \frac{1}{2}$ and $0\leq  t\leq \frac{1}{2}$.
\end{lemma}
\begin{proof}
Proof is similar as in Lemma \ref{P1Lemma8}.
\end{proof}
\begin{lemma}\label{P1Lemma9000}
Let $0<k<\pi^{2}$ and $u(t)\in C^{2}_{loc}\left((0,\frac{1}{2}],\mathbb{R}\right)$ be the solution of Problem $2(a)$, then 
\begin{equation}
u(t)=\int_{0}^{\frac{1}{2}}G(s,t)h(s)ds,
\end{equation}
where Green's function $G(s,t)$ is given by
\begin{equation}
  G(s,t)=\left\{
    \begin{array}{@{} l c @{}}
      -\frac{\cos\left(\frac{\sqrt{k}}{2}-\sqrt{k}t\right)\sin\left(\sqrt{k}s\right)}{\sqrt{k}\cos\left(\frac{\sqrt{k}}{2}\right)}&0\leq s \leq t,  \\
      -\frac{\cos\left(\frac{\sqrt{k}}{2}-\sqrt{k}s\right)\sin\left(\sqrt{k}t\right)}{\sqrt{k}\cos\left(\frac{\sqrt{k}}{2}\right)}&t\leq s \leq \frac{1}{2},
    \end{array}\right.
 % \label{P1Eq27}
\end{equation}
and $G(s,t)\leq 0$ for all $0\leq s\leq \frac{1}{2}$ and $0\leq  t\leq \frac{1}{2}$.
\end{lemma}
\begin{proof}
Proof is similar as in Lemma \ref{P1Lemma9}.
\end{proof}
\begin{lemma}\label{P1Lemma10000}
Let $0<k\leq\frac{\pi^{2}}{4}$, $\sqrt{k}\cos\left(\frac{\sqrt{k}}{2}\right)-\sin\left(\frac{\sqrt{k}}{2}\right)>0$ 
 and $u(t)\in C^{2}_{loc}\left((0,\frac{1}{2}],\mathbb{R}\right)$ be the solution of Problem $3(a)$, then 
\begin{equation}
u(t)=\int_{0}^{\frac{1}{2}}G(s,t)h(s)ds,
\end{equation}
where Green's function $G(s,t)$ is given by
\begin{equation}
  G(s,t)=\left\{
    \begin{array}{@{} l c @{}}
      -\frac{\left[\sqrt{k}\cos\left(\frac{\sqrt{k}}{2}-\sqrt{k}t\right)-\sin\left(\frac{\sqrt{k}}{2}-\sqrt{k}t\right)\right]  \sin\left(\sqrt{k}s\right)}{\sqrt{k}\left[\sqrt{k}\cos\left(\frac{\sqrt{k}}{2}\right)-\sin\left(\frac{\sqrt{k}}{2}\right)\right]}&0\leq s \leq t,  \\
      -\frac{\left[\sqrt{k}\cos\left(\frac{\sqrt{k}}{2}-\sqrt{k}s\right)-\sin\left(\frac{\sqrt{k}}{2}-\sqrt{k}s\right)\right]  \sinh\left(\sqrt{k}t\right)}{\sqrt{k}\left[\sqrt{k}\cos\left(\frac{\sqrt{k}}{2}\right)-\sin\left(\frac{\sqrt{k}}{2}\right)\right]}&t\leq s \leq \frac{1}{2},
    \end{array}\right. 
\end{equation}
and  $G(s,t)\leq 0$ for all $0\leq s\leq \frac{1}{2}$ and $0\leq  t\leq \frac{1}{2}$.
\end{lemma}
\begin{proof}
Proof is similar as in Lemma \ref{P1Lemma10}.
\end{proof}
\subsection{Anti-maximum principle}
\begin{proposition}\label{P1Proposition311} Let $k<0$ and $h(t)\in C^{2}_{loc}\left((0,\frac{1}{2}],\mathbb{R}\right)$ is such that $h(t)\geq0$, then the solutions of Problem $1(a)$ and Problem $2(a)$  are non positive. 
\end{proposition}
\begin{proposition} Let $k<0$, $\left(\sqrt{|k|}\cosh\left(\frac{\sqrt{|k|}}{2}\right)-\sinh\left(\frac{\sqrt{|k|}}{2}\right)\right)>0$ and $h(t)\in C^{2}_{loc}\left((0,\frac{1}{2}],\mathbb{R}\right)$ is such that $h(t)\geq0$, then the solutions of Problem $3(a)$ are non positive. 
\end{proposition}
\begin{proposition}\label{P1Proposition3110} Let $0<k<4\pi^{2}$ (respectively $0<k<\pi^{2}$) and $h(t)\in C^{2}_{loc}\left((0,\frac{1}{2}],\mathbb{R}\right)$ is such that $h(t)\geq0$, then the solutions of Problem $1(a)$ (respectively Problem $2(a)$) are non positive. 
\end{proposition}
\begin{proposition} Let $0<k\leq \frac{\pi^{2}}{4}$, $\left(\sqrt{k}\cos\left(\frac{\sqrt{k}}{2}\right)-\sin\left(\frac{\sqrt{k}}{2}\right)\right)>0$ and $h(t)\in C^{2}_{loc}\left((0,\frac{1}{2}],\mathbb{R}\right)$ is such that $h(t)\geq0$, then the solutions of Problem $3(a)$ are non positive. 
\end{proposition}
\subsection{Reverse order   lower and upper solutions}
Here, we define lower and upper solutions corresponding to Problem $1$, Problem $2$ and Problem $3$.
\begin{definition}(lower solution)
A function $\alpha \in  C^{2}_{loc}\left((0,\frac{1}{2}],\mathbb{R}\right)$ is the lower solution of Problem $1$ (respectively Problem $2$ and Problem $3$) if 
\begin{equation}\label{P1Eq36}
\alpha''\leq\frac{\alpha^{2}}{8t^{2}}+\frac{\lambda}{2},~\mbox{for}~t\in \left(0,\frac{1}{2}\right],
\end{equation}
with $\displaystyle \lim_{t\rightarrow 0}\frac{\alpha(t)}{\sqrt{t}}=0$ and $\alpha\left(\frac{1}{2}\right)\leq 0$ (respectively $\alpha'\left(\frac{1}{2}\right)\leq 0$ and $\alpha\left(\frac{1}{2}\right)\leq \alpha'\left(\frac{1}{2}\right)$).
\end{definition}
\begin{definition}(upper solution)
A function $\beta \in  C^{2}_{loc}\left((0,\frac{1}{2}],\mathbb{R}\right)$ is the upper solution of Problem $1$ (respectively Problem $2$ and Problem $3$) if 
\begin{equation}\label{P1Eq37}
\beta''\geq\frac{\beta^{2}}{8t^{2}}+\frac{\lambda}{2},~\mbox{for}~t\in \left(0,\frac{1}{2}\right].
\end{equation}
with $\displaystyle \lim_{t\rightarrow 0}\frac{\beta(t)}{\sqrt{t}}=0$ and $\beta\left(\frac{1}{2}\right)\geq 0$ (respectively $\beta'\left(\frac{1}{2}\right)\geq 0$ and $\beta\left(\frac{1}{2}\right)\geq \beta'\left(\frac{1}{2}\right)$).
\end{definition}
Now, we construct two sequences $\lbrace\alpha_{n}\rbrace$ and $\lbrace \beta_{n}\rbrace$ corresponding to Problem $1(a)$ (respectively Problem $2(a)$ and Problem $3(a)$), which are defined by
\begin{eqnarray}
&&\nonumber \alpha_{0}=\alpha,\\
\label{P1Eq40}
&&\alpha_{n+1}''+k\alpha_{n+1}=\frac{\alpha_{n}^{2}}{8t^{2}}+\frac{\lambda}{2}+k \alpha_{n},~\mbox{for}~t\in \left(0,\frac{1}{2}\right],\\
\label{P1Eq41}
&& \lim_{t\rightarrow 0}\frac{\alpha_{n+1}(t)}{\sqrt{t}}=0~\mbox{and}~ \alpha_{n+1}\left(\frac{1}{2}\right)= 0, 
\end{eqnarray}
$~\left(\mbox{respectively} ~\alpha_{n+1}'\left(\frac{1}{2}\right)= 0~ \mbox{and}~ \alpha_{n+1}\left(\frac{1}{2}\right)= \alpha_{n+1}'\left(\frac{1}{2}\right)\right)$ and 
\begin{eqnarray}
&&\nonumber \beta_{0}=\beta,\\
\label{P1Eq42}
&&\beta_{n+1}''+k\beta_{n+1}=\frac{\beta_{n}^{2}}{8t^{2}}+\frac{\lambda}{2}+k \beta_{n},~\mbox{for}~t\in \left(0,\frac{1}{2}\right],\\
\label{P1Eq43}
&& \lim_{t\rightarrow 0}\frac{\beta_{n+1}(t)}{\sqrt{t}}=0~\mbox{and}~ \beta_{n+1}\left(\frac{1}{2}\right)= 0, 
\end{eqnarray}
$\left(\mbox{respectively}~ \beta_{n+1}'\left(\frac{1}{2}\right)= 0~ \mbox{and}~ \beta_{n+1}\left(\frac{1}{2}\right)= \beta_{n+1}'\left(\frac{1}{2}\right)\right)$. We assume the following properties:
\begin{itemize}
\item $P_1$: $\alpha_{0}$ and $\beta_{0}$ satisfies 
\begin{eqnarray}
\lim_{t\rightarrow 0}\frac{|\alpha_{0}(t)|}{t}<\infty, ~ \lim_{t\rightarrow 0}\alpha_{0}(t)=0,~\alpha_{0}(t)\leq0,
\end{eqnarray}
and
\begin{eqnarray}
\lim_{t\rightarrow 0}\frac{|\beta_{0}(t)|}{t}<\infty, ~ \lim_{t\rightarrow 0}\beta_{0}(t)=0,
\end{eqnarray}
\item $P_2$: $h(t,u)$ is continuous on $D_{0}$ where $D_{0}=\left \lbrace (t,u)\in \left(0,\frac{1}{2}\right]\times \mathbb{R}: \beta_{0}=\beta\leq u \leq \alpha_{0} \right \rbrace$.
\end{itemize}
\begin{theorem}\label{P1Theorem1}
Assume $k<0$, $\lambda \in \mathbb{R}$ and there exist $\alpha_{0}$ and $\beta_{0}\in C^{2}_{loc}\left((0,\frac{1}{2}],\mathbb{R}\right)$ are lower and upper solutions of Problem $1$ which satisfy the properties $P_1$ and $P_2$  such that  $\beta_{0}\leq \alpha_{0}=0$,
then the  Problem $1$ has at least one solution in the region $D_{0}$ and the sequences $\lbrace\alpha_{n}\rbrace$, $\lbrace\beta_{n}\rbrace$ defined by (\ref{P1Eq40}), (\ref{P1Eq41}) and (\ref{P1Eq42}), (\ref{P1Eq43}) converges to a solutions $u$, $v$  uniformly as well as monotonically respectively, such that 
\begin{equation}
\beta\leq u \leq v \leq \alpha=0,~\forall t\in \left(0,\frac{1}{2}\right].
\end{equation}
\end{theorem}
\begin{proof}
We divide the proof into three parts. In the first part, we prove that 
\begin{equation}
\beta_{n} ~\mbox{is a upper solution of problem 1},~\beta_{n}\leq \beta_{n+1}~\mbox{and}~\beta_{n+1}\leq \alpha_{0}~\forall n\in \mathbb{N}.
\end{equation}
We apply mathematical induction on $n$. For $n=0$, from (\ref{P1Eq42}) and (\ref{P1Eq43}), we have
\begin{eqnarray}\label{P1Eq50}
&&\beta_{1}''+k\beta_{1}=\frac{\beta_{0}^{2}}{8t^{2}}+\frac{\lambda}{2}+k \beta_{0},~\mbox{for}~t\in \left(0,\frac{1}{2}\right],\\
\label{P1Eq51}
&& \lim_{t\rightarrow 0}\frac{\beta_{1}(t)}{\sqrt{t}}=0~\mbox{and}~ \beta_{1}\left(\frac{1}{2}\right)= 0. 
\end{eqnarray}
Now, from equation (\ref{P1Eq37}), we have
\begin{eqnarray}\label{P1Eq52}
&&\left(\beta_{0}-\beta_{1}\right)''+k\left(\beta_{0}-\beta_{1}\right)=-\frac{\beta_{0}^{2}}{8t^{2}}-\frac{\lambda}{2}+\beta_{0}''\geq 0,~\mbox{for}~t\in \left(0,\frac{1}{2}\right],\\
\label{P1Eq53}
&& \lim_{t\rightarrow 0}\frac{\beta_{0}-\beta_{1}(t)}{\sqrt{t}}=0~\mbox{and}~ \left(\beta_{0}-\beta_{1}\right)\left(\frac{1}{2}\right)\geq 0. 
\end{eqnarray}
Therefore by proposition \ref{P1Proposition311}, we have $\beta_{0}\leq \beta_{1}$. Again from (\ref{P1Eq36}) and (\ref{P1Eq50}), we have
\begin{eqnarray}\label{P1Eq54}
\left(\beta_{1}-\alpha_{0}\right)''+k\left(\beta_{1}-\alpha_{0}\right)&&=\frac{\beta_{0}^{2}}{8t^{2}}+\frac{\lambda}{2}+k\left(\beta_{0}-\alpha_{0}\right)-\alpha_{0}'',~\mbox{for}~t\in \left(0,\frac{1}{2}\right],\\
&&\geq \left(\frac{\beta_{0}+\alpha_{0}}{8t}+kt\right)\left(\frac{\beta_{0}-\alpha_{0}}{t}\right).
\end{eqnarray}
Since $\beta_{0}\leq \alpha_{0}$, therefore we have 
\begin{equation}
\left(\beta_{1}-\alpha_{0}\right)''+k\left(\beta_{1}-\alpha_{0}\right)\geq 0,~\forall t\in \left(0,\frac{1}{2}\right],
\end{equation} 
\begin{eqnarray}
\label{P1Eq55}
&& \lim_{t\rightarrow 0}\frac{(\beta_{1}-\alpha_{0})(t)}{\sqrt{t}}=0~\mbox{and}~ \left(\beta_{1}-\alpha_{0}\right)\left(\frac{1}{2}\right)\geq 0. 
\end{eqnarray}
Hence by proposition \ref{P1Proposition311}, we have $\beta_{1}\leq \alpha_{0}$. So our assumptions are true for $n=0$. Let our assumptions be true up to $n=m$. Therefore, we have
\begin{equation}\label{P5Eq13000}
\beta_{n} ~\mbox{is a upper solution of problem 1},~\beta_{n}\leq \beta_{n+1}~\mbox{and}~\beta_{n+1}\leq \alpha_{0}~\mbox{for}~ n=1,2,\cdots,m.
\end{equation} Now we want to show that our assumptions are true for $n+1$. Therefore from equation (\ref{P1Eq42}), we have
\begin{eqnarray}
\beta_{n+1}''-\frac{\beta_{n+1}^{2}}{8t^{2}}-\frac{\lambda}{2}&&=\frac{\beta_{n}^{2}-\beta_{n+1}^{2}}{8t^{2}}+k \left(\beta_{n}-\beta_{n+1}\right),~\mbox{for}~t\in \left(0,\frac{1}{2}\right],\\
&&\geq \left(\frac{\beta_{n}+\beta_{n+1}}{8t}+kt\right)\left(\frac{\beta_{n}-\beta_{n+1}}{t}\right).
\end{eqnarray}
Again by using conditions (\ref{P5Eq13000}), we have
\begin{eqnarray}\label{P1Eq61}
&&\beta_{n+1}''\geq \frac{\beta_{n+1}^{2}}{8t^{2}}+\frac{\lambda}{2},~\mbox{for}~t\in \left(0,\frac{1}{2}\right],\\
&&\lim_{t\rightarrow 0}\frac{\beta_{n+1}(t)}{\sqrt{t}}=0~\mbox{and}~ \beta_{n+1}\left(\frac{1}{2}\right)\geq 0.
\end{eqnarray}
Hence $\beta_{n+1}$ is a upper solution of Problem $1$. Now, from equation (\ref{P1Eq42}) and (\ref{P1Eq61}), we have
\begin{eqnarray}\label{P1Eq63}
&&\left(\beta_{n+1}-\beta_{n+2}\right)''+k\left(\beta_{n+1}-\beta_{n+2}\right)=-\frac{\beta_{n+1}^{2}}{8t^{2}}-\frac{\lambda}{2}+\beta_{n+1}''\geq 0,~\mbox{for}~t\in \left(0,\frac{1}{2}\right],\\
\label{P1Eq64}
&& \lim_{t\rightarrow 0}\frac{(\beta_{n+1}-\beta_{n+2})(t)}{\sqrt{t}}=0~\mbox{and}~ \left(\beta_{n+1}-\beta_{n+2}\right)\left(\frac{1}{2}\right)\geq 0. 
\end{eqnarray}
So by proposition \ref{P1Proposition311}, we have $\beta_{n+1}\leq \beta_{n+2}$. Again from (\ref{P1Eq36}) and (\ref{P1Eq42}), we have
\begin{eqnarray}\label{P1Eq65}
\left(\beta_{n+2}-\alpha_{0}\right)''+k\left(\beta_{n+2}-\alpha_{0}\right)&&=\frac{\beta_{n+1}^{2}}{8t^{2}}+\frac{\lambda}{2}+k\left(\beta_{n+1}-\alpha_{0}\right)-\alpha_{0}'',~\mbox{for}~t\in \left(0,\frac{1}{2}\right],\\
&&\geq \left(\frac{\beta_{n+1}+\alpha_{0}}{8t}+kt\right)\left(\frac{\beta_{n+1}-\alpha_{0}}{t}\right).
\end{eqnarray}
By similar analysis, we have $\beta_{n+2}\leq \alpha_{0}$. Hence by mathematical induction, we have
\begin{equation}
\beta_{n} ~\mbox{is a upper solution of Problem 1},~\beta_{n}\leq \beta_{n+1}~\mbox{and}~\beta_{n+1}\leq \alpha_{0}~\forall n\in \mathbb{N}.
\end{equation}
In the second part of the proof, we have to show 
\begin{equation}
\alpha_{n} ~\mbox{is a lower solution of Problem 1}~\mbox{and}~\alpha_{n+1}\leq \alpha_{n}~\forall n\in \mathbb{N}.
\end{equation}
Now from (\ref{P1Eq40}) and (\ref{P1Eq41}), we have
\begin{eqnarray}
\label{P1Eq66}
&&\alpha_{1}''+k\alpha_{1}=\frac{\alpha_{0}^{2}}{8t^{2}}+\frac{\lambda}{2}+k \alpha_{0},~\mbox{for}~t\in \left(0,\frac{1}{2}\right],\\
\label{P1Eq67}
&& \lim_{t\rightarrow 0}\frac{\alpha_{1}(t)}{\sqrt{t}}=0~\mbox{and}~ \alpha_{1}\left(\frac{1}{2}\right)= 0. 
\end{eqnarray}
Therefore, by using (\ref{P1Eq36}) we have
\begin{eqnarray}\label{P1Eq68}
\left(\alpha_{1}-\alpha_{0}\right)''+k\left(\alpha_{1}-\alpha_{0}\right)&&=\frac{\alpha_{0}^{2}}{8t^{2}}+\frac{\lambda}{2}-\alpha_{0}'',~\mbox{for}~t\in \left(0,\frac{1}{2}\right],\\
&&\geq 0.
\end{eqnarray}
Again,
\begin{eqnarray}
\lim_{t\rightarrow 0}\frac{(\alpha_{1}-\alpha_{0})(t)}{\sqrt{t}}=0~\mbox{and}~ (\alpha_{1}-\alpha_{0})\left(\frac{1}{2}\right)\geq 0. 
\end{eqnarray}
Hence by proposition \ref{P1Proposition311}, we have $\alpha_{1}\leq \alpha_{0}$. So our assumptions are true for $n=0$. Let our assumptions be true up to $n=m.$ So, we have
\begin{equation}
\alpha_{n} ~\mbox{is a lower solution of Problem 1}~\mbox{and}~\alpha_{n+1}\leq \alpha_{n}~\mbox{for} ~n=1,2,\cdots,m.
\end{equation}
Now, for $n+1$ we have
\begin{eqnarray}\label{P1Eq70}
\alpha_{n+1}''-\frac{\alpha_{n+1}^{2}}{8t^{2}}-\frac{\lambda}{2}&&=\frac{\alpha_{n}^{2}-\alpha_{n+1}^{2}}{8t^{2}}+k\left(\alpha_{n}-\alpha_{n+1}\right),~\mbox{for}~t\in \left(0,\frac{1}{2}\right],\\
&&\leq\left(\frac{1}{8} \frac{\alpha_{n}+\alpha_{n+1}}{t}+kt\right)\left(\frac{\alpha_{n}-\alpha_{n+1}}{t}\right),\\
&&\leq 0.
\end{eqnarray}
Therefore,
\begin{eqnarray}\label{P1Eq76}
\alpha_{n+1}''\leq \frac{\alpha_{n+1}^{2}}{8t^{2}}+\frac{\lambda}{2},~\mbox{for}~t\in \left(0,\frac{1}{2}\right],
\end{eqnarray}
and 
\begin{eqnarray}
\lim_{t\rightarrow 0}\frac{\alpha_{n+1}(t)}{\sqrt{t}}=0,~ \alpha_{n+1}\left(\frac{1}{2}\right)\leq 0. 
\end{eqnarray}
Hence, we have $\alpha_{n+1}$ is a lower solution of Problem $1$.
Therefore, by using (\ref{P1Eq76}), (\ref{P1Eq40}) and (\ref{P1Eq41}), we have
\begin{eqnarray}\label{P1Eq73}
\left(\alpha_{n+2}-\alpha_{n+1}\right)''+k\left(\alpha_{n+2}-\alpha_{n+1}\right)&&=\frac{\alpha_{n+1}^{2}}{8t^{2}}+\frac{\lambda}{2}-\alpha_{n+1}'',~\mbox{for}~t\in \left(0,\frac{1}{2}\right],\\
&&\geq 0.
\end{eqnarray}
and 
\begin{eqnarray}
\lim_{t\rightarrow 0}\frac{(\alpha_{n+2}-\alpha_{n+1})(t)}{\sqrt{t}}=0,~ (\alpha_{n+2}-\alpha_{n+1})\left(\frac{1}{2}\right)\geq 0. 
\end{eqnarray}
Therefore by proposition \ref{P1Proposition311}, we have $\alpha_{n+2}\leq \alpha_{n+1}$. Hence by mathematical induction we conclude that
\begin{equation}
\alpha_{n} ~\mbox{is a lower solution of Problem 1}~\mbox{and}~\alpha_{n+1}\leq \alpha_{n}~\forall n\in \mathbb{N}.
\end{equation}
In the last part of the proof, we want to show $\beta_{n}\leq \alpha_{n}$ for all $n\in \mathbb{N}$. Again from (\ref{P1Eq61}) and (\ref{P1Eq76}), we have
\begin{eqnarray}\label{P1Eq74}
\left(\beta_{n+1}-\alpha_{n+1}\right)''+k\left(\beta_{n+1}-\alpha_{n+1}\right)&&=\frac{\beta_{n}^{2}}{8t^{2}}+\frac{\lambda}{2}+k\left(\beta_{n}-\alpha_{n}\right)-\alpha_{n}'',~\mbox{for}~t\in \left(0,\frac{1}{2}\right],\\
&&\geq \left(\frac{\beta_{n}+\alpha_{n}}{8t}+kt\right)\left(\frac{\beta_{n}-\alpha_{n}}{t}\right).
\end{eqnarray}
Since $\beta_{n}\leq \alpha_{n}\leq0$, therefore we have 
\begin{equation}
\left(\beta_{n+1}-\alpha_{n+1}\right)''+k\left(\beta_{n+1}-\alpha_{n+1}\right)\geq 0,
\end{equation} 
and
\begin{eqnarray}
\label{P1Eq75}
&& \lim_{t\rightarrow 0}\frac{\beta_{n+1}-\alpha_{n+1}(t)}{\sqrt{t}}=0,~ \left(\beta_{n+1}-\alpha_{n+1}\right)\left(\frac{1}{2}\right)\geq 0. 
\end{eqnarray}
Hence by proposition \ref{P1Proposition311}, we have $\beta_{n+1}\leq \alpha_{n+1}$. Finally we have 
\begin{equation}\label{P1Eq001}
\beta=\beta_{0}\leq \beta_{1}\leq \cdots\leq \beta_{n}\leq\cdots\leq\alpha_{n}\leq\cdots\leq\alpha_{1}\leq\alpha_{0}=0.
\end{equation}
Let $t_{n}\in \left(0,\frac{1}{2}\right)$ for $n\in \mathbb{N}$ such that
\begin{equation}
t_{n+1}<t_{n}~\mbox{for $n\in \mathbb{N}$},~\lim_{n\rightarrow +\infty}t_{n}=0.
\end{equation}
Therefore, for every $n\in \mathbb{N}$ there exists  a solution $\alpha_{n}$ and $\beta_{n}$ to equations (\ref{P1Eq40}), (\ref{P1Eq41}) and (\ref{P1Eq42}), (\ref{P1Eq43}) respectively  satisfy the inequality (\ref{P1Eq001}) on the interval $[t_{n},\frac{1}{2}]$. Since $\lbrace\alpha_{n}\rbrace$ and $\lbrace \beta_{n} \rbrace$ are monotone and bounded, therefore they converge to function $u(t)$ and $v(t)$ respectively. Therefore, by Dini's theorem we have,  there exists  $u(t) $ and $v(t)$  such that 
\begin{equation}
\lim_{n\rightarrow \infty} \alpha_{n}=u~\mbox{and}~\lim_{n\rightarrow \infty} \beta_{n}=v~\mbox{uniformly on every compact interval $\left[t_n,\frac{1}{2}\right]$ of $\left(0, \frac{1}{2}\right]$}.
\end{equation}
Hence, from (\ref{P1Eq40}), (\ref{P1Eq41}), (\ref{P1Eq42}), (\ref{P1Eq43}) and (\ref{P10002}), we have there exists solutions $v(t)\in C^{2}_{loc}\left((0,\frac{1}{2}],\mathbb{R}\right)$ and $u(t)\in C^{2}_{loc}\left((0,\frac{1}{2}],\mathbb{R}\right)$ to Problem $1$ satisfying
\begin{equation}
\beta\leq u \leq v \leq \alpha_{0}=0,~\forall t\in \left(0,\frac{1}{2}\right].
\end{equation}
 Hence the proof is complete.
\end{proof}
\begin{theorem}\label{P1Theorem2}
Assume $k<0$, $\lambda \in \mathbb{R}$ and there exist $\alpha_{0}$ and $\beta_{0}\in C^{2}_{loc}\left((0,\frac{1}{2}],\mathbb{R}\right)$ are lower and upper solutions of Problem $2$ which satisfy the properties $P_1$ and $P_2$ such that $\beta_{0}\leq \alpha_{0}=0$, then the  Problem $2$ has at least one solution in the region $D_{0}$ and the sequences $\lbrace\alpha_{n}\rbrace$, $\lbrace\beta_{n}\rbrace$ defined by (\ref{P1Eq40}), (\ref{P1Eq41}) and (\ref{P1Eq42}), (\ref{P1Eq43}) converges to a solutions $u$, $v$  uniformly as well as monotonically respectively, such that 
\begin{equation}
\beta\leq u \leq v \leq \alpha=0,~\forall t\in \left(0,\frac{1}{2}\right].
\end{equation}
\end{theorem}
\begin{proof}
The proof is same as in Theorem \ref{P1Theorem1}.
\end{proof}
\begin{theorem}\label{P1theorem3}
Assume $k<0$, $\lambda \in \mathbb{R}$, $\left(\sqrt{|k|}\cosh\left(\frac{\sqrt{|k|}}{2}\right)-\sinh\left(\frac{\sqrt{|k|}}{2}\right)\right)>0$ and there exist $\alpha_{0}$ and $\beta_{0}\in C^{2}_{loc}\left((0,\frac{1}{2}],\mathbb{R}\right)$ are lower and upper solutions of Problem $3$ which satisfy the properties $P_1$ and $P_2$ such that $\beta_{0}\leq \alpha_{0}=0$, then the  Problem $3$ has at least one solution in the region $D_{0}$ and the sequences $\lbrace\alpha_{n}\rbrace$, $\lbrace\beta_{n}\rbrace$ defined by (\ref{P1Eq40}), (\ref{P1Eq41}) and (\ref{P1Eq42}), (\ref{P1Eq43}) converges to a solutions $u$, $v$  uniformly as well as monotonically respectively, such that 
\begin{equation}
\beta\leq u \leq v \leq \alpha=0,~\forall t\in \left(0,\frac{1}{2}\right].
\end{equation}
\end{theorem}
\begin{proof}
The proof is same as in Theorem \ref{P1Theorem1}.
\end{proof}
\begin{theorem}
 Let  $\alpha_{0}$, $\beta_{0}\in C^{2}_{loc}\left((0,\frac{1}{2}],\mathbb{R}\right)$ are the lower and upper solutions of Problem $1$ which satisfy the properties $P_1$ and $P_2$ such that $\beta_{0}\leq \alpha_{0}$. Assume $0<k<k'$, where $\displaystyle k'=\min\left \lbrace  4\pi^{2}, -\max_{t\in (0, \frac{1}{2}]}\frac{\alpha_{0}}{2t}\right \rbrace$ and $\lambda \in \mathbb{R}$. Then  the  Problem $1$ has at least one solution in the region $D_{0}$ and the sequences $\lbrace\alpha_{n}\rbrace$, $\lbrace\beta_{n}\rbrace$ defined by (\ref{P1Eq40}), (\ref{P1Eq41}) and (\ref{P1Eq42}), (\ref{P1Eq43}) converges to a solutions $u$, $v$  uniformly as well as monotonically respectively, such that 
\begin{equation}
\beta\leq u \leq v \leq \alpha,~\forall t\in \left(0,\frac{1}{2}\right].
\end{equation}
\end{theorem}
\begin{proof}
The proof is same as in Theorem \ref{P1Theorem1}.
\end{proof}
\begin{theorem}
Let  $\alpha_{0}$, $\beta_{0}\in C^{2}_{loc}\left((0,\frac{1}{2}],\mathbb{R}\right)$ are  the lower and upper solutions of Problem $2$ which satisfy the properties $P_1$ and $P_2$ such that $\beta_{0}\leq \alpha_{0}$. Assume $0<k<k'$, where $\displaystyle k'=\min\left \lbrace  \pi^{2}, -\max_{t\in (0, \frac{1}{2}]}\frac{\alpha_{0}}{2t}\right \rbrace$ and $\lambda \in \mathbb{R}$. Then the Problem $2$ has at least one solution in the region $D_{0}$ and the sequences $\lbrace\alpha_{n}\rbrace$, $\lbrace\beta_{n}\rbrace$ defined by (\ref{P1Eq40}), (\ref{P1Eq41}) and (\ref{P1Eq42}), (\ref{P1Eq43}) converges to a solutions $u$, $v$  uniformly as well as monotonically respectively, such that 
\begin{equation}
\beta\leq u \leq v \leq \alpha,~\forall t\in \left(0,\frac{1}{2}\right].
\end{equation}
\end{theorem}
\begin{proof}
The proof is same as in Theorem \ref{P1Theorem1}.
\end{proof}
\begin{theorem}
Let $\alpha_{0}$, $\beta_{0}\in C^{2}_{loc}\left((0,\frac{1}{2}],\mathbb{R}\right)$ are the lower and upper solutions of Problem $3$ which satisfy the properties $P_1$ and $P_2$ such that $\beta_{0}\leq \alpha_{0}$. Assume $\lambda \in \mathbb{R}$, $\sqrt{k}\cos\left(\frac{\sqrt{k}}{2}\right)-\sin\left(\frac{\sqrt{k}}{2}\right)>0$ and $0<k<k'$, where $\displaystyle k'=\min\left \lbrace  \frac{\pi^{2}}{4}, -\max_{t\in (0, \frac{1}{2}]}\frac{\alpha_{0}}{2t}\right \rbrace$. Then the  Problem $3$ has at least one solution in the region $D_{0}$ and the sequences $\lbrace\alpha_{n}\rbrace$, $\lbrace\beta_{n}\rbrace$ defined by (\ref{P1Eq40}), (\ref{P1Eq41}) and (\ref{P1Eq42}), (\ref{P1Eq43}) converges to a solutions $u$, $v$  uniformly as well as monotonically respectively, such that 
\begin{equation}
\beta\leq u \leq v \leq \alpha,~\forall t\in \left(0,\frac{1}{2}\right].
\end{equation}
\end{theorem}
\begin{proof}
The proof is same as in Theorem \ref{P1Theorem1}.
\end{proof}
\section{Estimations of $\lambda$}\label{P5Estimation}
The objective of this section is to derive some qualitative bounds of the parameter $\lambda$, from which we can conclude about the nonexistence of  solutions. The equation (\ref{P1Eq1113}) can be  written as in the following form:
\begin{equation}\label{P1Eq100}
(tu'-u)'=\frac{u^{2}}{8t}+\frac{\lambda t}{2},~\forall t\in \left(0,\frac{1}{2}\right].
\end{equation}
Put $\displaystyle v(t)=-\frac{u(t)}{t}$ and integrating from $0$ to $t$,  the equation (\ref{P1Eq100}) becomes
\begin{equation}\label{P1Eq101}
v'(t)=-\frac{1}{8t^{2}}\int_{0}^{t}v^2(s)s~ds-\frac{\lambda}{4},~\forall t\in \left(0,\frac{1}{2}\right].
\end{equation}
Therefore, we have
\begin{equation}\label{P1Eq102}
v(t)\geq0,~\forall t\in \left(0,\frac{1}{2}\right].
\end{equation}
In view of the transformation,  the boundary condition at $r=1$ becomes 
\begin{equation}\label{P1Eq103}
\mbox{BC of Problem $1$:}~v\left(\frac{1}{2}\right)=0,
\end{equation}
\begin{equation}\label{P1Eq104}
\mbox{BC of Problem $2$:}~v\left(\frac{1}{2}\right)=-\frac{1}{2} v\left(\frac{1}{2}\right),
\end{equation}
and
\begin{equation}\label{P1Eq105}
\mbox{BC of Problem $3$:}~v\left(\frac{1}{2}\right)=- v\left(\frac{1}{2}\right).
\end{equation}
Carlos et. al. in \cite{Carlos2012} prove the following two lemmas. 
\begin{lemma}\label{P1Lemma11}
The set of numbers $\lambda\geq 0$, for which there exists a solution $u(t)\in C^{2}_{loc}\left((0,\frac{1}{2}],\mathbb{R}\right)$ of equation (\ref{P1Eq1113}) satisfying $\displaystyle \lim_{t\rightarrow 0}\frac{u(t)}{\sqrt{t}}=0$ and $u(t)\leq0$, is nonempty and bounded from above.
\end{lemma}
\begin{lemma}\label{P1Lemma12}
If the Problem $1$, Problem $2$ and Problem $3$ are solvable for some $\lambda_{0}\geq 0$, then these are solvable for every $0\leq \lambda\leq \lambda_{0}$.
\end{lemma}
We present the following new results.
\begin{lemma}\label{P1Lemma13}
Let there exist a function $u\in C^{2}_{loc}\left((0,\frac{1}{2}],\mathbb{R}\right)$ satisfying equations (\ref{P1Eq101}), (\ref{P1Eq102}) and (\ref{P1Eq104}), then 
\begin{equation}\label{P1Eq106}
\lambda \leq \frac{384}{11}\approx 34.91.
\end{equation}
\end{lemma}
\begin{proof}
Now from equation (\ref{P1Eq101}), we have
\begin{equation}\label{P1Eq107}
v'(t)\leq 0,~\forall t\in \left(0,\frac{1}{2}\right].
\end{equation}
Again from equation (\ref{P1Eq101}), we get
\begin{equation}\label{P1Eq108}
v''(t)=\frac{1}{4t^{3}}\int_{0}^{t}v^2(s)s~ds-\frac{v^{2}(t)}{8t},~\forall t\in \left(0,\frac{1}{2}\right].
\end{equation}
Therefore by using (\ref{P1Eq107}) and (\ref{P1Eq102}), from (\ref{P1Eq108}) we have
\begin{equation}
\label{P1Eq109}
v''(t)\geq0,~\forall t\in \left(0,\frac{1}{2}\right].
\end{equation}
Therefore $v'(t)$ is increasing in $ \left(0,\frac{1}{2}\right].$ Now
\begin{equation}\label{P1Eq110}
v'(t)\leq v'\left(\frac{1}{2}\right)=-\frac{1}{2}\int_{0}^{\frac{1}{2}}v^2(s)s~ds-\frac{\lambda}{4},~\forall t\in \left(0,\frac{1}{2}\right].
\end{equation}
Therefore, we have 
\begin{equation}
\label{P1Eq111}
v'(t)\leq -c,~\forall t\in \left(0,\frac{1}{2}\right],
\end{equation}
where 
\begin{equation}\label{P1Eq112}
c=\frac{1}{2}\int_{0}^{\frac{1}{2}}v^2(s)s~ds+\frac{\lambda}{4}.
\end{equation}
Now, integrating equation (\ref{P1Eq111}) from $0$ to $t$ and by using equation (\ref{P1Eq104}), we have
\begin{equation}
\label{P1Eq113}
v(t)\geq c(1-t),~\forall t\in \left(0,\frac{1}{2}\right].
\end{equation}
Therefore, from equations (\ref{P1Eq112}) and (\ref{P1Eq113}), we get
\begin{equation}
\frac{11}{384}c^{2}-c+\frac{\lambda}{4}\leq 0,
\end{equation}
which implies the equation (\ref{P1Eq106}).
\end{proof}
\begin{lemma}\label{P1Lemma14}
Let \begin{equation}\label{P1Eq115}
0\leq \lambda \leq 2C~\mbox{and}~C\leq \frac{128}{9},
\end{equation} then there exists a  solution $\beta \in C^{2}_{loc}\left((0,\frac{1}{2}],\mathbb{R}\right)$ satisfies equation (\ref{P1Eq37}), the assumption $P_1$,  $\beta'\left(\frac{1}{2}\right)=0$ and $\beta(t)\leq0$.
\end{lemma}
\begin{proof}
We put 
\begin{equation}
\beta(t)=-C t (A-\sqrt{2t}),~\forall t\in \left(0,\frac{1}{2}\right]~\mbox{and}~C \geq 0.
\end{equation}
Obviously $\beta(t)$ satisfy assumption $P_1$. Now, $\beta'\left(\frac{1}{2}\right)=0$  implies $\displaystyle A=\frac{3}{2}$. Therefore $\beta(t)\leq0$ is also fulfilled. Now, we have
\begin{eqnarray}
&&\beta''(t)-\frac{\beta^{2}(t)}{8 t^{2}}-\frac{\lambda}{2}\\
&&=\frac{3C}{2 \sqrt{2t}}-\frac{C^{2}t^{2}\left(\frac{3}{2}-\sqrt{2t}\right)^{2}}{8 t^{2}}-\frac{\lambda}{2},\\
&&=\frac{C}{\sqrt{2t}}\left(\frac{3}{2}-\frac{\lambda\sqrt{2t}}{2C}\right)-\frac{C^{2}\left(\frac{3}{2}-\sqrt{2t}\right)^{2}}{8},\\
&&\geq\frac{C}{\sqrt{2t}}\left(\frac{3}{2}-\sqrt{2t}\right)-\frac{C^{2}\left(\frac{3}{2}-\sqrt{2t}\right)^{2}}{8},~\mbox{since $\lambda\leq 2C$}\\
&&=\frac{C^2}{8\sqrt{2t}}\left(\frac{3}{2}-\sqrt{2t}\right)\left(\left(\sqrt{2t}-\frac{3}{4}\right)^{2}+\frac{-9C+128}{16C}\right),\\
&&\geq0,~\forall t\in \left(0,\frac{1}{2}\right],~ \mbox{since $\displaystyle C\leq \frac{128}{9}$}.
\end{eqnarray}
Hence the inequality (\ref{P1Eq37}) is satisfied.
\end{proof}
\begin{lemma}\label{P1Lemma14000}
Let \begin{equation}\label{P1Eq115000}
0\leq \lambda \leq 3C~\mbox{and}~C\leq 48,
\end{equation} then there exists a  solution $\beta \in C^{2}_{loc}\left((0,\frac{1}{2}],\mathbb{R}\right)$ satisfies equation (\ref{P1Eq37}), the assumption $P_1$,  $\beta\left(\frac{1}{2}\right)=0$ and $\beta(t)\leq0$.
\end{lemma}
\begin{proof}
We put \begin{equation}
\beta(t)=-C t (A-\sqrt{2t}),~\forall t\in \left(0,\frac{1}{2}\right]~\mbox{and}~C \geq 0.
\end{equation} 
Again, $\beta(t)$ satisfy assumption $P_1$. Now, $\beta\left(\frac{1}{2}\right)=0$  implies $\displaystyle A=1$. Hence, $\beta(t)\leq0$ is also fulfilled. Now, we have
\begin{eqnarray}
&&\beta''(t)-\frac{\beta^{2}(t)}{8 t^{2}}-\frac{\lambda}{2}\\
&&=\frac{3C}{2 \sqrt{2t}}-\frac{C^{2}t^{2}\left(1-\sqrt{2t}\right)^{2}}{8 t^{2}}-\frac{\lambda}{2},\\
&&=\frac{3C}{2\sqrt{2t}}\left(1-\frac{\lambda \sqrt{2t}}{3C}\right)-\frac{C^{2}\left(1-\sqrt{2t}\right)^{2}}{8},\\
&&\geq\frac{3C}{2\sqrt{2t}}\left(1-\sqrt{2t}\right)-\frac{C^{2}\left(1-\sqrt{2t}\right)^{2}}{8},~\mbox{since $\lambda\leq 3C$}\\
&&=\frac{C^2}{8\sqrt{2t}}\left(1-\sqrt{2t}\right)\left(\left(\sqrt{2t}-\frac{1}{2}\right)^{2}+\frac{-C+48}{4C}\right),\\
&&\geq0,~\forall t\in \left(0,\frac{1}{2}\right],~ \mbox{since $\displaystyle C\leq 48$}.
\end{eqnarray} 
This completes the proof.
\end{proof}
\begin{lemma}\label{P1Lemma15000}
Let \begin{equation}\label{P1Eq116000}
0\leq \lambda \leq \frac{3C}{2}~\mbox{and}~C\leq 6,
\end{equation} then there exists a  solution $\beta \in C^{2}_{loc}\left((0,\frac{1}{2}],\mathbb{R}\right)$ satisfies equation (\ref{P1Eq37}), the assumption $P_1$,  $\beta\left(\frac{1}{2}\right)=\beta'\left(\frac{1}{2}\right)$ and $\beta(t)\leq0$.
\end{lemma}
\begin{proof}
We put \begin{equation}
\beta(t)=-C t (A-\sqrt{2t}),~\forall t\in \left(0,\frac{1}{2}\right]~\mbox{and}~C \geq 0.
\end{equation}
Now, $\beta(t)$ also satisfy assumption $P_1$. Similarly, $\beta\left(\frac{1}{2}\right)=\beta'\left(\frac{1}{2}\right)$  implies $\displaystyle A=2$. So, $\beta(t)\leq0$ is also fulfilled. Therefore, we have
\begin{eqnarray}
&&\beta''(t)-\frac{\beta^{2}(t)}{8 t^{2}}-\frac{\lambda}{2}\\
&&=\frac{3C}{2 \sqrt{2t}}-\frac{C^{2}t^{2}\left(2-\sqrt{2t}\right)^{2}}{8 t^{2}}-\frac{\lambda}{2},\\
&&=\frac{3C}{4\sqrt{2t}}\left(2-\frac{2\lambda \sqrt{2t}}{3C}\right)-\frac{C^{2}\left(2-\sqrt{2t}\right)^{2}}{8},\\
&&\geq\frac{3C}{4\sqrt{2t}}\left(2-\sqrt{2t}\right)-\frac{C^{2}\left(2-\sqrt{2t}\right)^{2}}{8},~\mbox{since $\lambda\leq \frac{3C}{2}$}\\
&&=\frac{C^2}{8\sqrt{2t}}\left(2-\sqrt{2t}\right)\left(\left(\sqrt{2t}-1\right)^{2}+\frac{-C+6}{C}\right),\\
&&\geq0,~\forall t\in \left(0,\frac{1}{2}\right], ~\mbox{since $\displaystyle C\leq 6$}.
\end{eqnarray} 
Hence, the proof is complete.
\end{proof}
\begin{theorem}\label{P1Theroem4}
Let $\lambda_{0}\in \mathbb{R^{+}}$. If $0\leq\lambda <\lambda_{0}$, then the equation (\ref{P5Eq1050}) corresponding to different types of boundary condition are solvable. Also there is no solution of these problems if $\lambda>\lambda_{0}$. Furthermore, every solutions $w(r)$ of governing equation corresponding to these three types of boundary conditions satisfy
\begin{equation}\label{P1Eq120}
w(r)\leq 0, ~r\in (0,1]~\mbox{and}~\lim_{r\rightarrow 0^{+}}w(r)=0.
\end{equation}
\end{theorem}
\begin{proof}
The proof of this can be deduced from Lemma \ref{P1Lemma11}, Lemma \ref{P1Lemma12}, Lemma \ref{P1Lemma1}, Lemma \ref{P1Lemma2}, Lemma \ref{P1Lemma3} and Lemma \ref{P1Lemma4}.
\end{proof}
\begin{proposition}
Corresponding to equations (\ref{P5Intro10}) and (\ref{P5Eq205020}) the value of $\lambda_{0}$ admits the estimates 
\begin{equation}\label{P1Eq122000}
144\leq\lambda_{0}\leq 307.
\end{equation}
\begin{proof}
From  Lemma $7.7$ in \cite{Carlos2012} and Lemma \ref{P1Lemma14000}, we get the equation (\ref{P1Eq122000}).
\end{proof}
\end{proposition}
\begin{proposition}
Corresponding to equations (\ref{P5Intro10}) and (\ref{P5Eq205022}) the value of $\lambda_{0}$ admits the estimates 
\begin{equation}\label{P1Eq122}
\frac{256}{9}\leq\lambda_{0}\leq \frac{384}{11}.
\end{equation}
\end{proposition}
\begin{proof}
From Lemma \ref{P1Lemma13} and Lemma \ref{P1Lemma14}, we have the equation (\ref{P1Eq122}).
\end{proof}
\begin{proposition}
Corresponding to equations (\ref{P5Intro10}) and (\ref{P5Eq205021}) the value of $\lambda_{0}$ admits the estimates 
\begin{equation}\label{P1Eq123000}
9\leq\lambda_{0}\leq 11.63.
\end{equation}
\end{proposition}
\begin{proof}
By using the Lemma $7.6$ in \cite{Carlos2012}  and Lemma \ref{P1Lemma15000}, we have the equation (\ref{P1Eq123000}).
\end{proof}
\section{Numerical results and discussion}\label{P5ADM}
To find the approximate solutions, we develop the iterative numerical schemes with the help of the Fredholm integral equations (\ref{P5problem1'}), (\ref{P5problem2'}) and (\ref{P5problem3'}) respectively. Now, we decompose the solution $u(t)$ of the form
%\begin{equation}\label{P5ADM2}
$u(t)=\sum_{i=0}^{\infty} u_{i}(t)$, and
%\end{equation}
approximate the nonlinear term in terms of Adomian's polynomials (\cite{Ghorbani}) which is  given by
\begin{equation}\label{P5ADM3}
N\left(u(t)\right)=-\frac{1}{2}u^{2}(t)=\sum_{i=0}^{\infty} A_{i}(u_{0},u_{1},\cdots,u_{i}),
\end{equation}
where 
\begin{equation}\label{P5ADM6}
A_{i}=\frac{1}{i!} \frac{d^{i}}{d\beta^{i}}N\left(\sum_{j=0}^{i} \beta^{j} u_{j}\right)_{\beta=0},~i=0,1,2,\cdots.
\end{equation}
Therefore from integral equation (\ref{P5problem1'}), we define
\begin{equation}\label{P5problem1'''}
  \mbox{Scheme of Problem $1$}=\left\{
    \begin{array}{@{} l c @{}}
       u_{0}(t)=-c t-\frac{\lambda}{4}t\left(\frac{1}{2}-t\right),  \\
       \vdots\\
      u_{n+1}(t)=\int_{0}^{t}\left(\frac{s}{2}-\frac{t}{2}\right)\frac{A_{n}}{2 s^{2}}ds,\\
      \vdots,\\
      \mbox{and}~c=-\int_{0}^{\frac{1}{2}}\sum_{i=0}^{n}\frac{A_{i}}{2 s^{2}}\left(\frac{1}{2}-s\right)ds,
    \end{array}\right.
 % \label{P1Eq27}
\end{equation}
We compute the arbitary constant $c$ by using Mathematica software. For better understanding, we present below the algorithm of our proposed technique corresponding to equation (\ref{P5problem1'}).

\textbf{Algorithm:}

 Step $1.$ Convert Fredholm integral equation (\ref{P5problem1'})  into Voltera integral equation.

 Step $2.$ Identify the constant term, and approximate the nonlinear term by equation (\ref{P5ADM3}).
 
 Step $3.$ Consider $u_0(t)$ as in (\ref{P5problem1'''}), and obtain $u_i(t)$ for $i=1,2,\cdots,n+1$.
 
 Step $4.$ Approximate the term $\displaystyle \frac{u^2}{4 s^{2}}$ by $\displaystyle -\sum_{i=0}^{n}\frac{A_i}{2 s^{2}}$ in the equation $\displaystyle c=\int_{0}^{\frac{1}{2}}\frac{u^2}{4 s^{2}}\left(\frac{1}{2}-s\right)ds$.
 
 Step $5.$ Compute the values of the constant and the approximate solutions $u(t)=\displaystyle\sum_{i=0}^{n+1}u_i$.

Again, we apply the above algorithm on equations (\ref{P5problem2'}) and (\ref{P5problem3'}), and we define  the following iterative schemes:
\begin{equation}\label{P5problem2'''}
  \mbox{Scheme of Problem $2$}=\left\{
    \begin{array}{@{} l c @{}}
      u_{0}(t)=-c t-\frac{\lambda}{4}t\left(1-t\right),  \\
       \vdots\\
      u_{n+1}(t)=\int_{0}^{t}\left(s-t\right)\frac{A_{n}}{4 s^{2}}ds,\\
      \vdots,\\
      \mbox{and}~c=-\int_{0}^{\frac{1}{2}}\sum_{i=0}^{n}\frac{A_{i}}{4 s^{2}}ds,
    \end{array}\right.
 % \label{P1Eq27}
\end{equation}
\begin{equation}\label{P5problem3'''}
  \mbox{and Scheme of Problem $3$}=\left\{
    \begin{array}{@{} l c @{}}
      u_{0}(t)=-c t-\frac{\lambda}{4}t\left(\frac{3}{2}-t\right),  \\
       \vdots\\
      u_{n+1}(t)=\int_{0}^{t}\left(\frac{s}{2}-\frac{t}{2}\right)\frac{A_{n}}{2 s^{2}}ds,\\
      \vdots,\\
      \mbox{and}~c=-\int_{0}^{\frac{1}{2}}\sum_{i=0}^{n}\left(\frac{1}{2}+s\right)\frac{A_{i}}{2 s^{2}}ds.
    \end{array}\right.
 % \label{P1Eq27}
\end{equation}
Approximate solutions of equations (\ref{P5problem2'}) and (\ref{P5problem3'}) can be written as $\displaystyle u(t)=\sum_{i=0}^{n+1}u_{i}(t)$, provided the series is convergent for $n\rightarrow \infty$. Recently, the convergence of ADM was established by Amit Kumar Verma et. al. in \cite{Biswajit2019}. Now by using the transformation $\displaystyle t=\frac{r^{2}}{2}$, $u(t)=w(r)$, $w(r)=r \phi'(r)$ and $\phi(1)=0$,  we get the solutions of equation (\ref{P5Intro10}). We arrive at two cases:

\noindent{\em \textbf{Case (a)}: $\lambda \geq 0.$}\label{P1Casea}\\
For $\lambda=0$, we get one trivial and one non trivial solutions. For $0<\lambda \leq \lambda_{\text{critical}}$, we always find two non-trivial solutions. We may refer them as upper and lower solution respectively. Corresponding to equations (\ref{P5Eq205022}),  (\ref{P5Eq205021}) and (\ref{P5Eq205020}),  we find the  critical value of $\lambda$, i.e. $\lambda_{\text{critical}}$, is  to be $31.94$, $11.34$ and $168.76$  respectively. For $\lambda> \lambda_{\text{critical}}$, we do not find any numerical solutions as the value of $c$ become imaginary. In subsection \ref {P5Tables} and \ref{P5Figures},  we tabulate residual errors and provide approximate solutions graph corresponding to some positive $\lambda$.

\noindent{\em \textbf{Case (b)}: $\lambda < 0.$}\\
In this case, we always have two nontrivial numerical solutions corresponding to three types of boundary conditions. One solution is negative (namely the negative solution) and the other solution is positive (namely the positive solution). We do not find any negative critical $\lambda$. We place residue errors and approximate solutions graph in the next two subsections.
\subsection{Tables}\label{P5Tables}
Here, we have placed below some numerical data of approximate solutions of $\phi(r)$ corresponding to different types of boundary conditions. In table \ref{P5table1}, we see that for $\lambda=0$ the maximum absolute residue error of lower and upper solutions are $0$ and $8.34 \times 10^{-07}$ respectively. But for $\lambda=31.94$, we observe that the maximum absolute residue error of lower and upper solutions are $1.53 \times 10^{-14}$ and $1.35 \times 10^{-14}$ respectively. Therefore, if we are increasing the value of $\lambda$, we see that the residue error of the lower solution is increasing and the residue error of the upper solution is decreasing. Similarly, if we are decreasing the value of negative $\lambda$, we see that the residue error of both positive and negative solutions are decreasing ([see: table \ref{P5table2}]). Same remarks are made to tables \ref{P5table3}, \ref{P5table4}, \ref{P5table5} and \ref{P5table6}.

\begin{table}[H]										
\caption{\small{Residue error of approximate solutions $\phi(r)$ corresponding to boundary conditions (\ref{P5Eq205022}):}}									
\centering											
\begin{center}											
\resizebox{9.7cm}{2.8cm}{											
\begin{tabular}{c c  c  c  c    }											
\hline		

			& \multicolumn{2}{c}{Lower solution} 						&	 \multicolumn{2}{c}{Upper solution} 				\\\hline
			$r\setminus\lambda$	&	$0$	&	$31.94$	&	$0$	&	$31.94$	\\\hline
0	&	0	&	0	&	0	&	0	\\
0.1	&	0	&	1.11022E-16	&	0	&	-6.93889E-17	\\
0.2	&	0	&	-3.33067E-16	&	-2.66454E-15	&	-1.11022E-16	\\
0.3	&	0	&	2.22045E-16	&	0	&	-4.44089E-16	\\
0.4	&	0	&	1.33227E-15	&	-1.06581E-14	&	-1.33227E-15	\\
0.5	&	0	&	2.22045E-15	&	-3.01981E-14	&	-2.66454E-15	\\
0.6	&	0	&	0	&	-9.9476E-14	&	-5.32907E-15	\\
0.7	&	0	&	7.54952E-15	&	-2.23821E-13	&	-1.77636E-15	\\
0.8	&	0	&	1.53211E-14	&	2.49862E-11	&	1.35447E-14	\\
0.9	&	0	&	1.95399E-14	&	8.3347E-07	&	1.86517E-14	\\\hline
\end{tabular}}											
\end{center}											
\label{P5table1}											
\end{table}
\begin{table}[H]										
\caption{\small{Residue error of approximate solutions $\phi(r)$ corresponding to  boundary conditions (\ref{P5Eq205022}):}}											
\centering											
\begin{center}											
\resizebox{11.2cm}{3cm}{											
\begin{tabular}{c c  c  c  c    }											
\hline		

			& \multicolumn{2}{c}{Positive solution} 						&	 \multicolumn{2}{c}{Negative solution} 				\\\hline
			$r\setminus\lambda$	&	$-1$	&	$-15$	&	$-1$	&	$-15$	\\\hline
0	&	0	&	0	&	0	&	0	\\
0.1	&	-7.21645E-16	&	2.22045E-16	&	1.34339E-18	&	-3.64292E-17	\\
0.2	&	4.44089E-16	&	4.44089E-15	&	-2.57498E-18	&	-5.20417E-17	\\
0.3	&	0	&	-1.77636E-15	&	-8.23994E-18	&	5.55112E-17	\\
0.4	&	0	&	-8.88178E-15	&	-2.60209E-18	&	-2.77556E-17	\\
0.5	&	1.24345E-14	&	-3.19744E-14	&	1.38778E-17	&	3.05311E-16	\\
0.6	&	6.03961E-14	&	-2.02505E-13	&	-1.82146E-17	&	1.94289E-16	\\
0.7	&	1.30562E-13	&	-3.53495E-13	&	-9.84456E-17	&	-3.33067E-16	\\
0.8	&	4.3471E-11	&	3.47169E-08	&	1.37856E-16	&	5.68989E-16	\\
0.9	&	1.42341E-06	&	0.001199979	&	3.68629E-17	&	3.66374E-15	\\\hline

\end{tabular}}											
\end{center}											
\label{P5table2}											
\end{table}

\vspace{-.4cm}
\begin{table}[H]										
\caption{\small{Residue error of approximate solutions $\phi(r)$ corresponding to boundary conditions (\ref{P5Eq205021}):}}											
\centering											
\begin{center}											
\resizebox{10cm}{2.9cm}{											
\begin{tabular}{c c  c  c  c    }											
\hline		

			& \multicolumn{2}{c}{Lower solution} 						&	 \multicolumn{2}{c}{Upper solution} 				\\\hline
			$r\setminus\lambda$	&	$0$	&	$11.34$	&	$0$	&	$11.34$	\\\hline
0	&	0	&	0	&	0	&	0	\\
0.1	&	0	&	-1.21431E-16	&	8.32667E-17	&	4.51028E-17	\\
0.2	&	0	&	-5.55112E-17	&	-2.22045E-16	&	-1.38778E-16	\\
0.3	&	0	&	-1.66533E-16	&	8.88178E-16	&	-5.55112E-17	\\
0.4	&	0	&	1.11022E-16	&	-8.88178E-16	&	-2.22045E-16	\\
0.5	&	0	&	-4.44089E-16	&	-2.66454E-15	&	4.44089E-16	\\
0.6	&	0	&	2.22045E-15	&	-2.66454E-15	&	4.44089E-16	\\
0.7	&	0	&	3.55271E-15	&	3.55271E-15	&	2.66454E-15	\\
0.8	&	0	&	3.55271E-15	&	8.88178E-16	&	1.33227E-15	\\
0.9	&	0	&	3.33067E-15	&	-1.77636E-15	&	1.9984E-15	\\\hline

\end{tabular}}											
\end{center}											
\label{P5table3}											
\end{table}
\begin{table}[H]										
\caption{\small{Residue error of approximate solutions $\phi(r)$ corresponding to boundary conditions (\ref{P5Eq205021}):}}											
\centering											
\begin{center}											
\resizebox{11.2cm}{2.9cm}{											
\begin{tabular}{c c  c  c  c    }											
\hline		

			& \multicolumn{2}{c}{Positive solution} 						&	 \multicolumn{2}{c}{Negative solution} 				\\\hline
			$r\setminus\lambda$	&	$-1$	&	$-15$	&	$-1$	&	$-15$	\\\hline
0	&	0	&	0	&	0	&	0	\\
0.1	&	-2.77556E-17	&	-4.44089E-16	&	3.04932E-18	&	-3.98986E-17	\\
0.2	&	1.11022E-16	&	-6.66134E-16	&	-5.96311E-18	&	-1.73472E-16	\\
0.3	&	4.44089E-16	&	0	&	7.04731E-18	&	2.63678E-16	\\
0.4	&	8.88178E-16	&	0	&	2.42861E-17	&	0	\\
0.5	&	2.66454E-15	&	0	&	-2.55872E-17	&	1.11022E-16	\\
0.6	&	3.55271E-15	&	-3.55271E-15	&	-2.34188E-17	&	6.66134E-16	\\
0.7	&	1.77636E-14	&	-7.10543E-15	&	3.64292E-17	&	-3.33067E-16	\\
0.8	&	3.10862E-14	&	-7.63833E-14	&	1.38778E-17	&	-3.33067E-16	\\
0.9	&	6.83897E-14	&	-4.59188E-13	&	2.54571E-16	&	7.77156E-16	\\\hline

\end{tabular}}											
\end{center}											
\label{P5table4}											
\end{table}
\begin{table}[H]										
\caption{\small{Residue error of approximate solutions $\phi(r)$ corresponding to  boundary conditions (\ref{P5Eq205020}):}}											
\centering											
\begin{center}											
\resizebox{9.7cm}{2.7cm}{											
\begin{tabular}{c c  c  c  c    }											
\hline		

			& \multicolumn{2}{c}{Lower solution} 						&	 \multicolumn{2}{c}{Upper solution} 				\\\hline
			$r\setminus\lambda$	&	$0$	&	$168.76$	&	$0$	&	$168.76$	\\\hline
0	&	0	&	0	&	0	&	0	\\
0.1	&	0	&	8.88178E-16	&	-4.44089E-16	&	4.44089E-16	\\
0.2	&	0	&	1.77636E-15	&	-1.77636E-15	&	4.44089E-16	\\
0.3	&	0	&	1.24345E-14	&	-1.06581E-14	&	-7.10543E-15	\\
0.4	&	0	&	3.19744E-14	&	-5.32907E-15	&	7.10543E-15	\\
0.5	&	0	&	6.4615E-14	&	-3.81917E-14	&	1.33227E-15	\\
0.6	&	0	&	9.23706E-14	&	-2.66454E-14	&	-3.55271E-14	\\
0.7	&	0	&	3.69482E-13	&	4.98552E-10	&	-9.23706E-14	\\
0.8	&	0	&	6.1668E-11	&	0.000557377	&	1.70296E-10	\\
0.9	&	0	&	6.71694E-06	&	-	&	1.96355E-05	\\\hline

\end{tabular}}											
\end{center}											
\label{P5table5}											
\end{table}
\begin{table}[H]										
\caption{\small{Residue error of approximate solutions $\phi(r)$ corresponding to boundary conditions (\ref{P5Eq205020}):}}											
\centering											
\begin{center}											
\resizebox{11.2cm}{3cm}{											
\begin{tabular}{c c  c  c  c    }											
\hline		

			& \multicolumn{2}{c}{Positive solution} 						&	 \multicolumn{2}{c}{Negative solution} 				\\\hline
			$r\setminus\lambda$	&	$-1$	&	$-15$	&	$-1$	&	$-15$	\\\hline
0	&	0	&	0	&	0	&	0	\\
0.1	&	8.88178E-16	&	2.22045E-16	&	6.0097E-19	&	-1.27936E-17	\\
0.2	&	3.55271E-15	&	3.55271E-15	&	-1.999E-18	&	5.55112E-17	\\
0.3	&	5.32907E-15	&	5.32907E-15	&	2.6834E-18	&	-2.25514E-17	\\
0.4	&	2.84217E-14	&	5.68434E-14	&	6.66784E-18	&	-5.55112E-17	\\
0.5	&	8.08242E-14	&	1.38556E-13	&	2.92328E-17	&	-3.71231E-16	\\
0.6	&	1.63425E-13	&	1.03029E-13	&	1.85941E-17	&	6.2797E-16	\\
0.7	&	4.99185E-10	&	4.83531E-10	&	7.08526E-17	&	-1.38778E-17	\\
0.8	&	0.000557832	&	0.000562261	&	-7.29668E-17	&	5.82867E-16	\\
0.9	&	 -	&	 -	&	1.17636E-16	&	2.41474E-15	\\\hline

\end{tabular}}											
\end{center}											
\label{P5table6}											
\end{table}

\subsection{Figures}\label{P5Figures}
Here, we have displayed few graphs corresponding to three types of boundary conditions. For positive values of $\lambda$, we see that two solutions are moving to each other for increasing the value of $\lambda$. For critical value of $\lambda$, we do not find the unique solution numerically. For negative values of $\lambda$, we observe that two nontrivial solutions are moving away from each other for decreasing the value of $\lambda$. 
\begin{figure}[H]
\centering
\subfigure[\,\,$\lambda=0$]{\includegraphics[width=.5\linewidth]{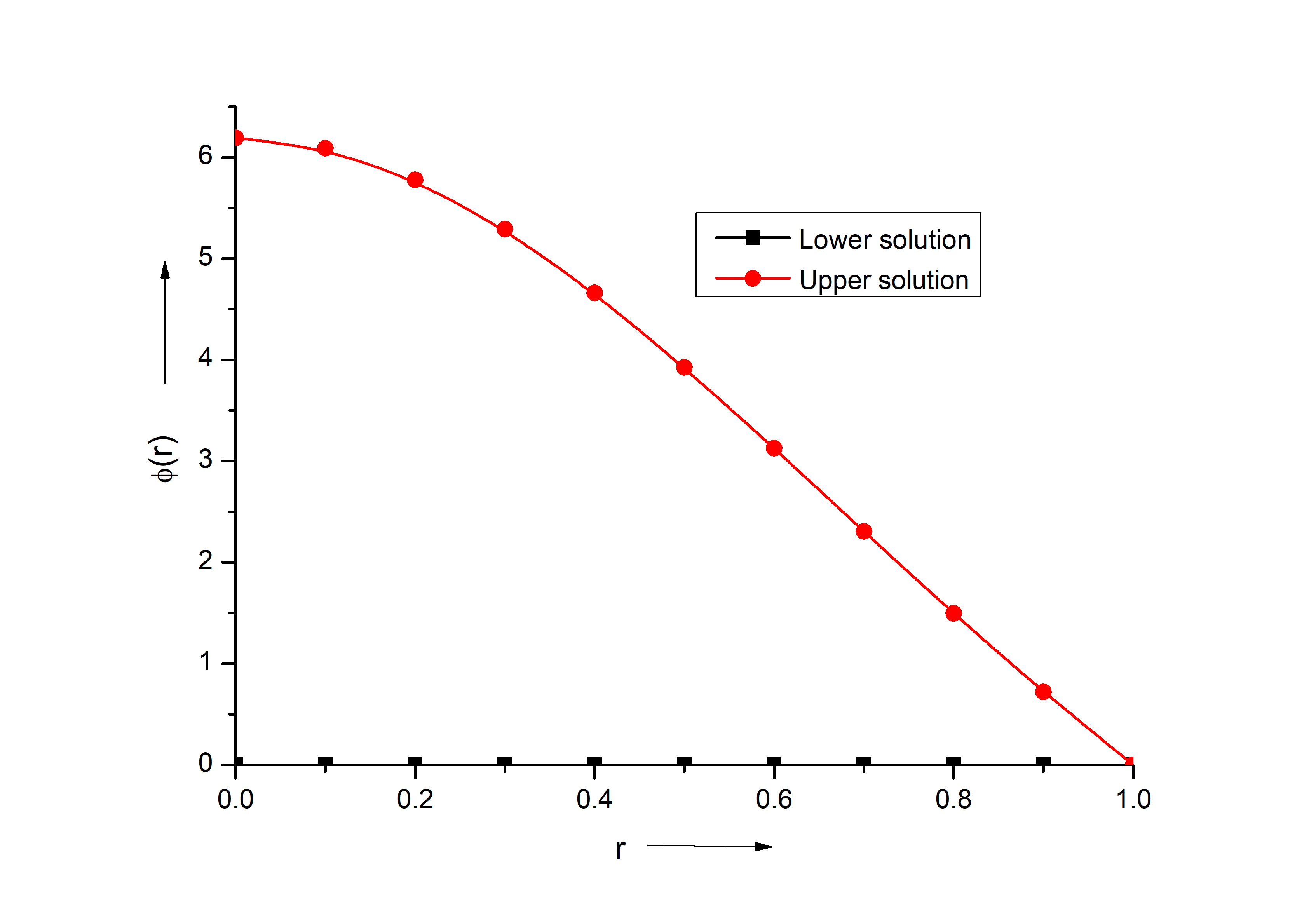}} \hspace{-.6cm}
\subfigure[\,\,$\lambda=31.94$]{\includegraphics[width=.5\linewidth]{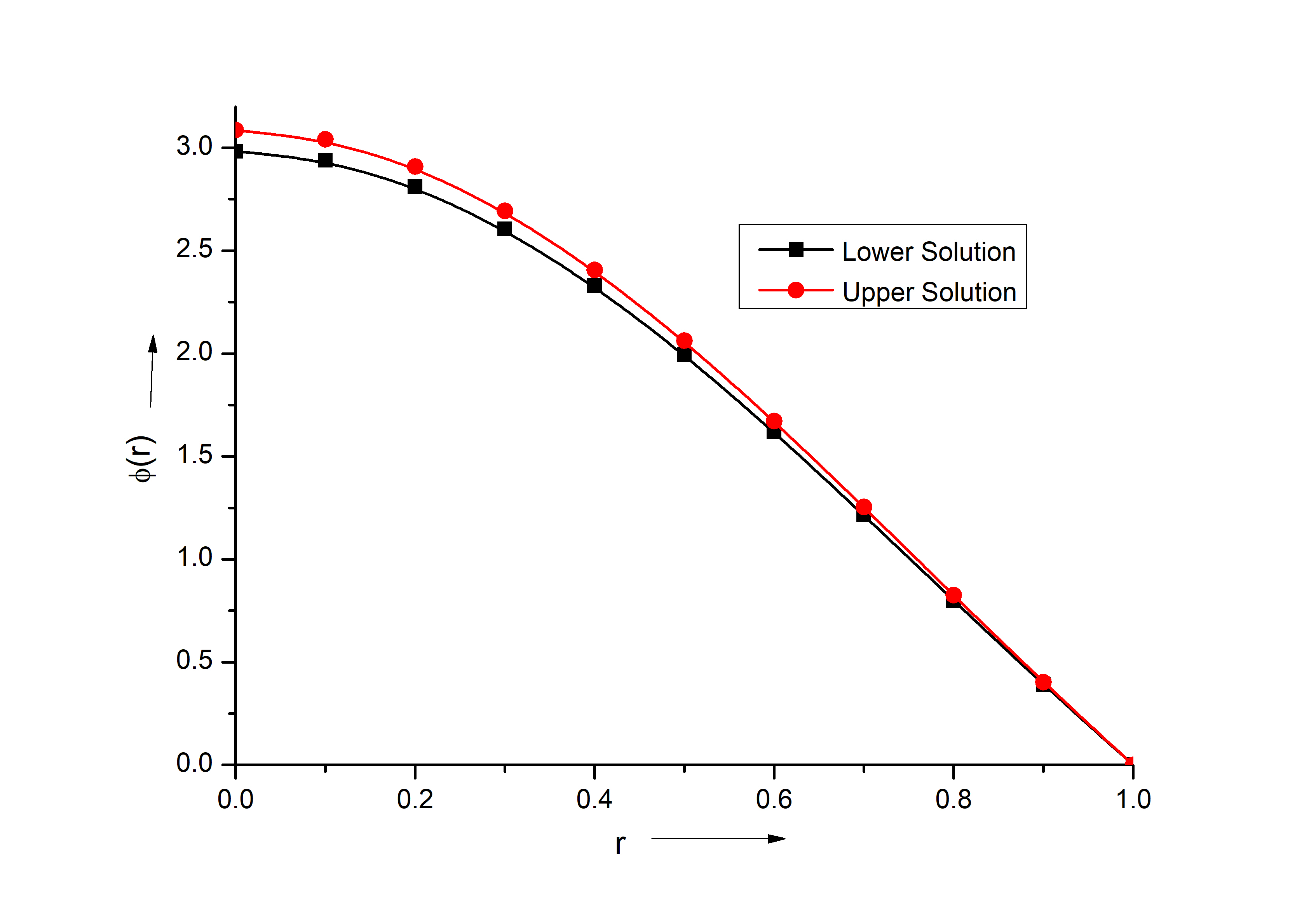}}
\caption{Approximate solutions $\phi(r)$ corresponding to equations (\ref{P5Intro10}) and (\ref{P5Eq205022}).}
\label{P5AG11}
\end{figure}
\begin{figure}[H]
\centering
\subfigure[\,\,$\lambda=-1$]{\includegraphics[width=.5\linewidth]{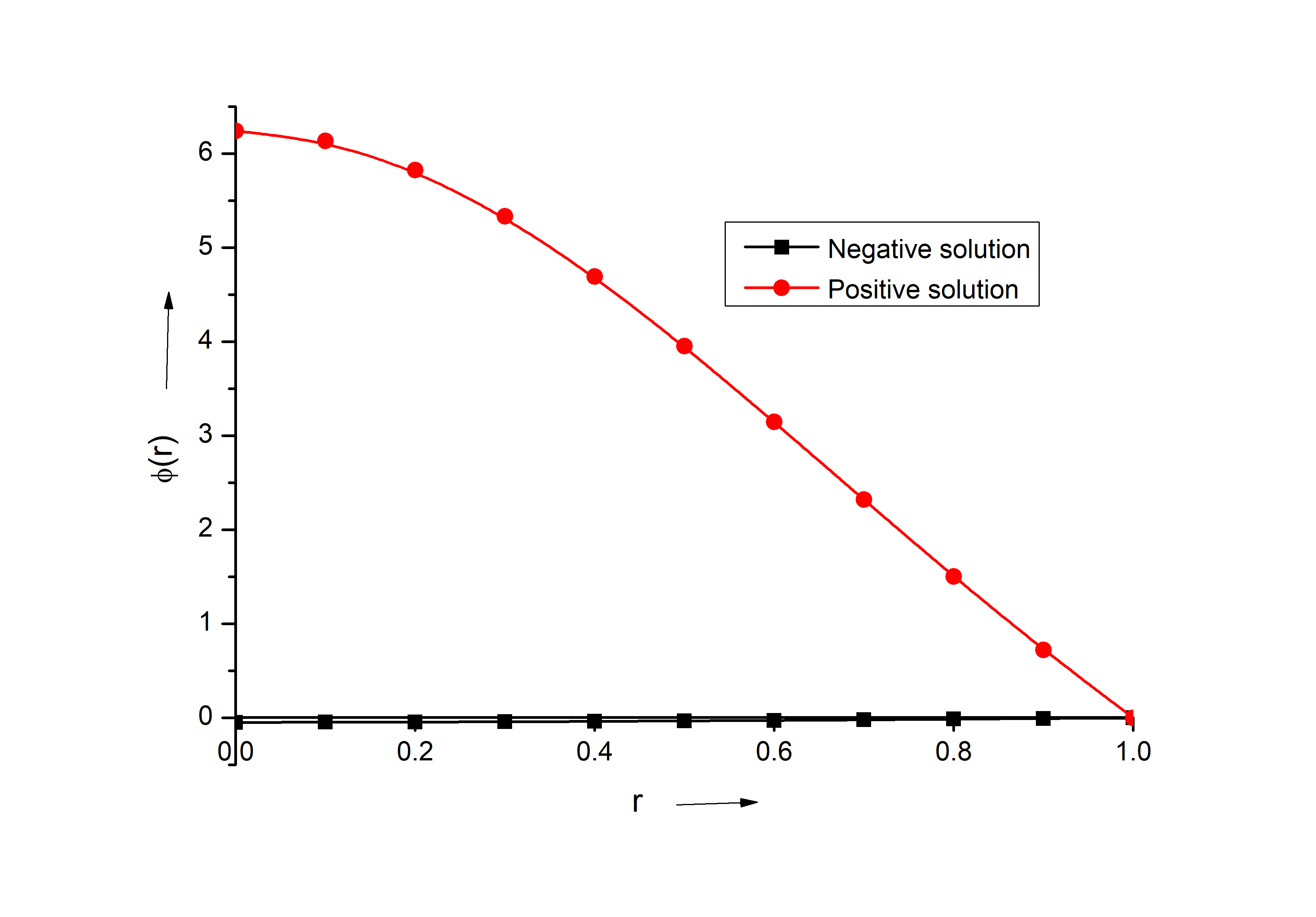}} \hspace{-.6cm}
\subfigure[\,\,$\lambda=-15$]{\includegraphics[width=.5\linewidth]{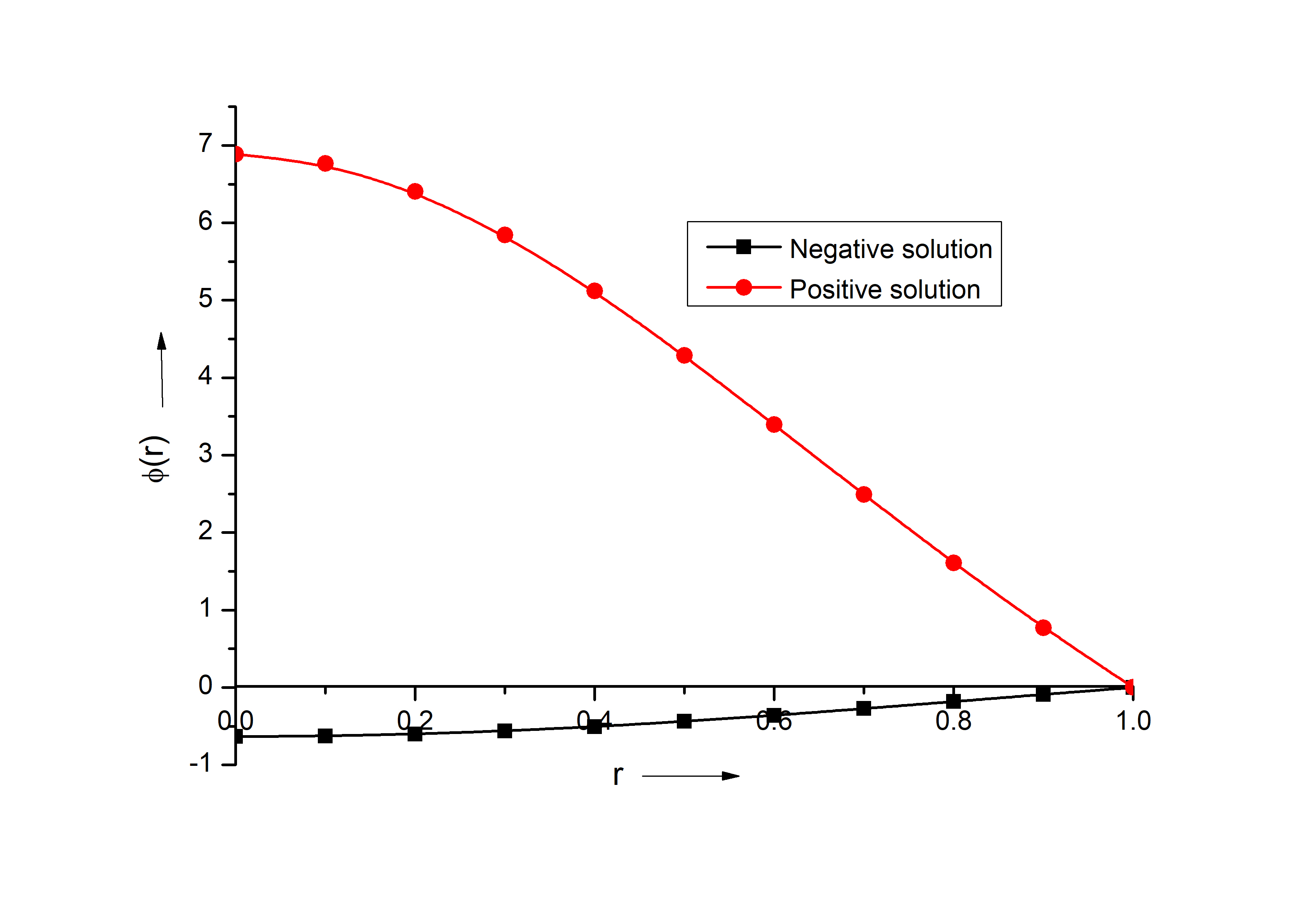}}
\caption{Approximate solutions $\phi(r)$ corresponding to equations (\ref{P5Intro10}) and (\ref{P5Eq205022}).}
\label{P5AG22}
\end{figure}
\begin{figure}[H]
\centering
\subfigure[\,\,$\lambda=0$]{\includegraphics[width=.5\linewidth]{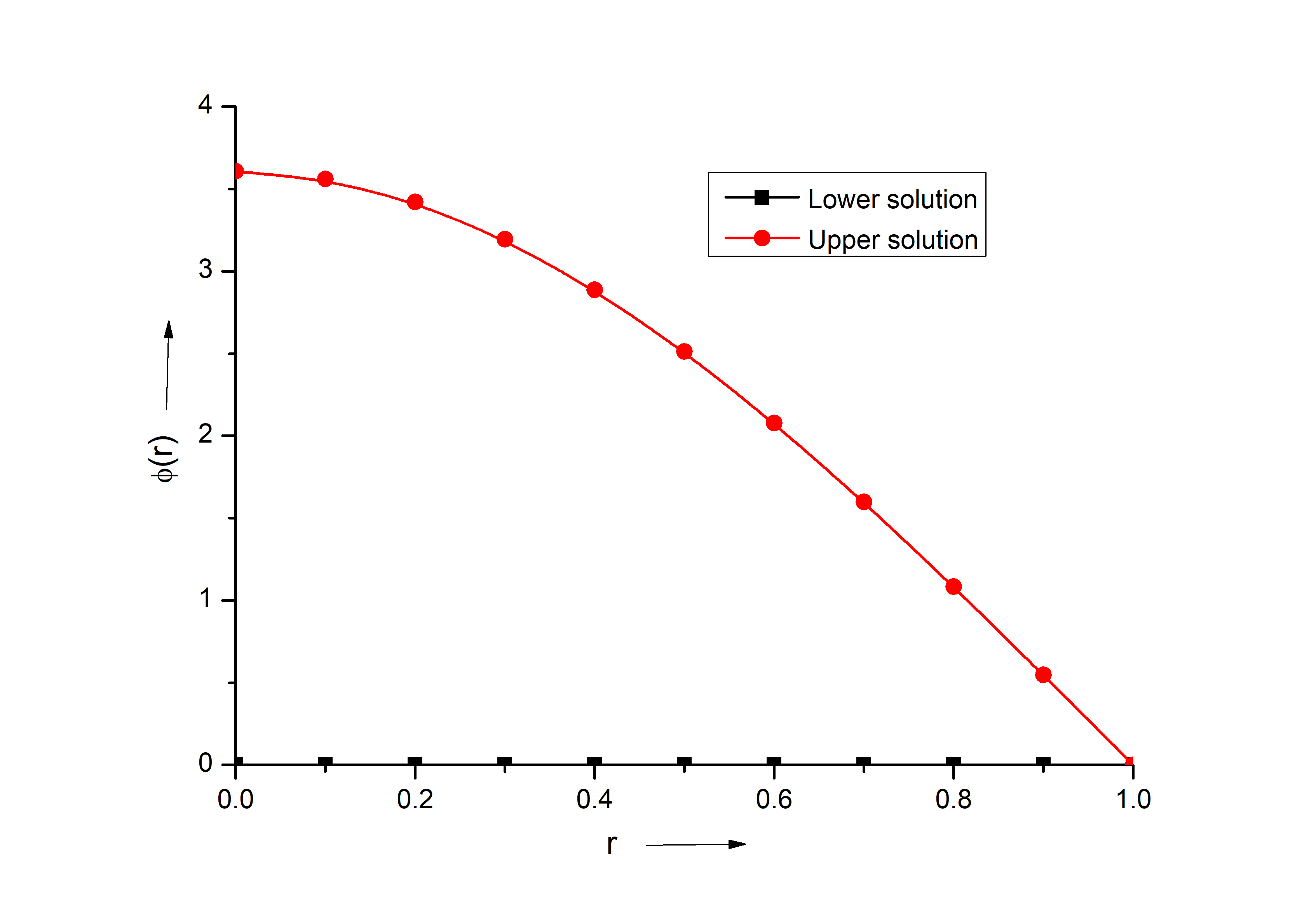}} \hspace{-.6cm}
\subfigure[\,\,$\lambda=11.34$]{\includegraphics[width=.5\linewidth]{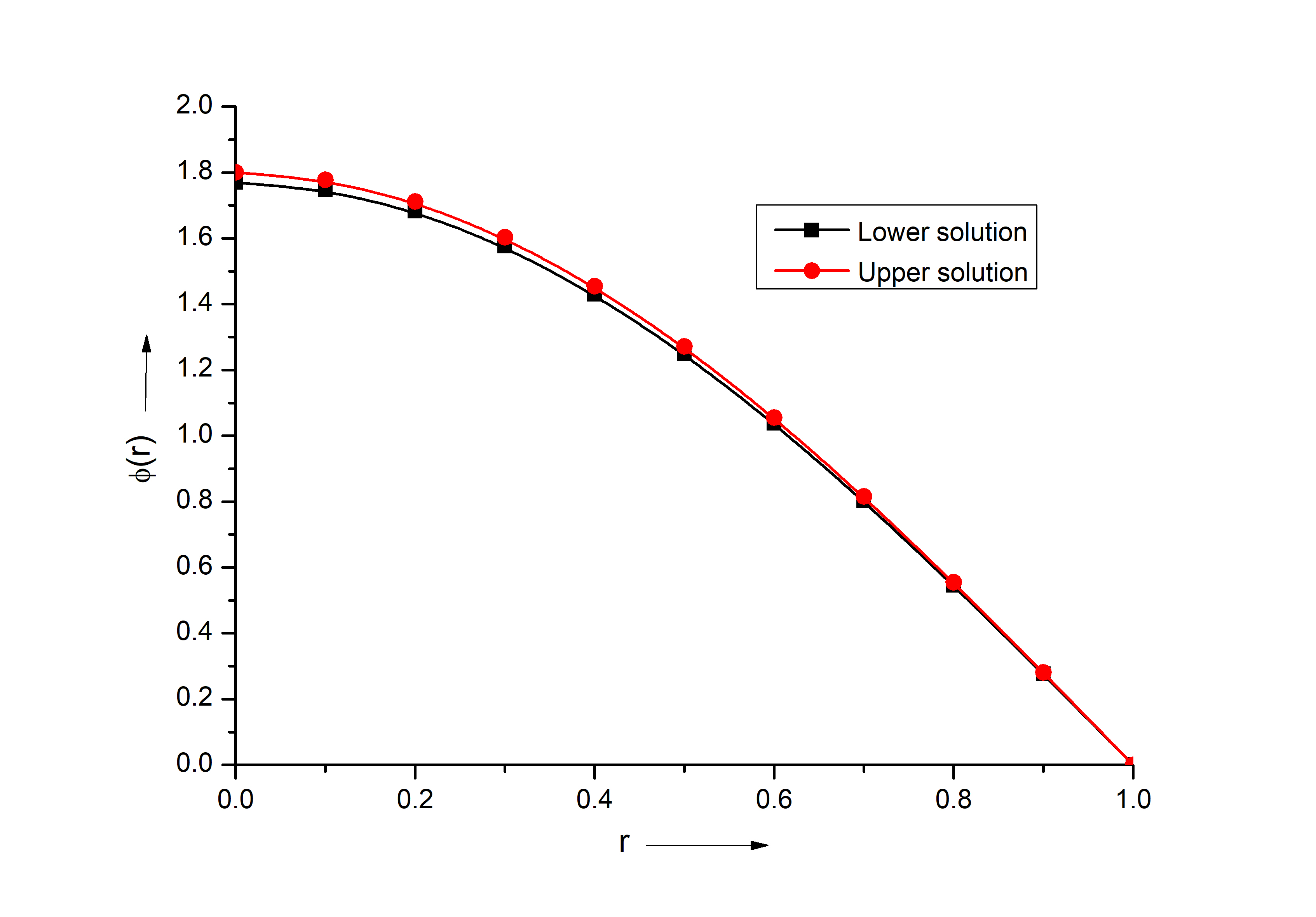}}
\caption{Approximate solutions $\phi(r)$ corresponding to equations (\ref{P5Intro10}) and (\ref{P5Eq205021}).}
\label{P5AG33}
\end{figure}
\begin{figure}[H]
\centering
\subfigure[\,\,$\lambda=-1$]{\includegraphics[width=.5\linewidth]{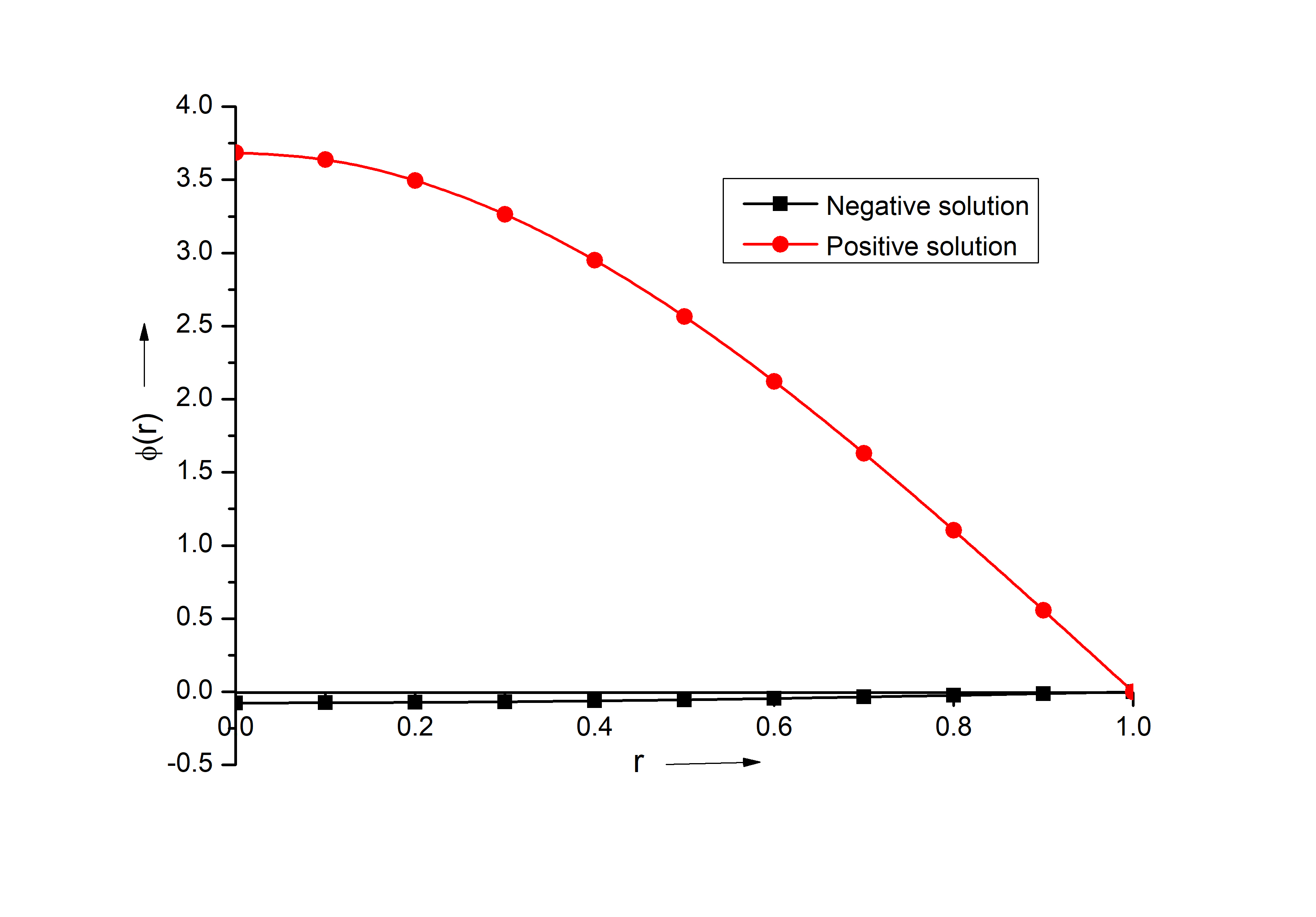}} \hspace{-.6cm}
\subfigure[\,\,$\lambda=-15$]{\includegraphics[width=.5\linewidth]{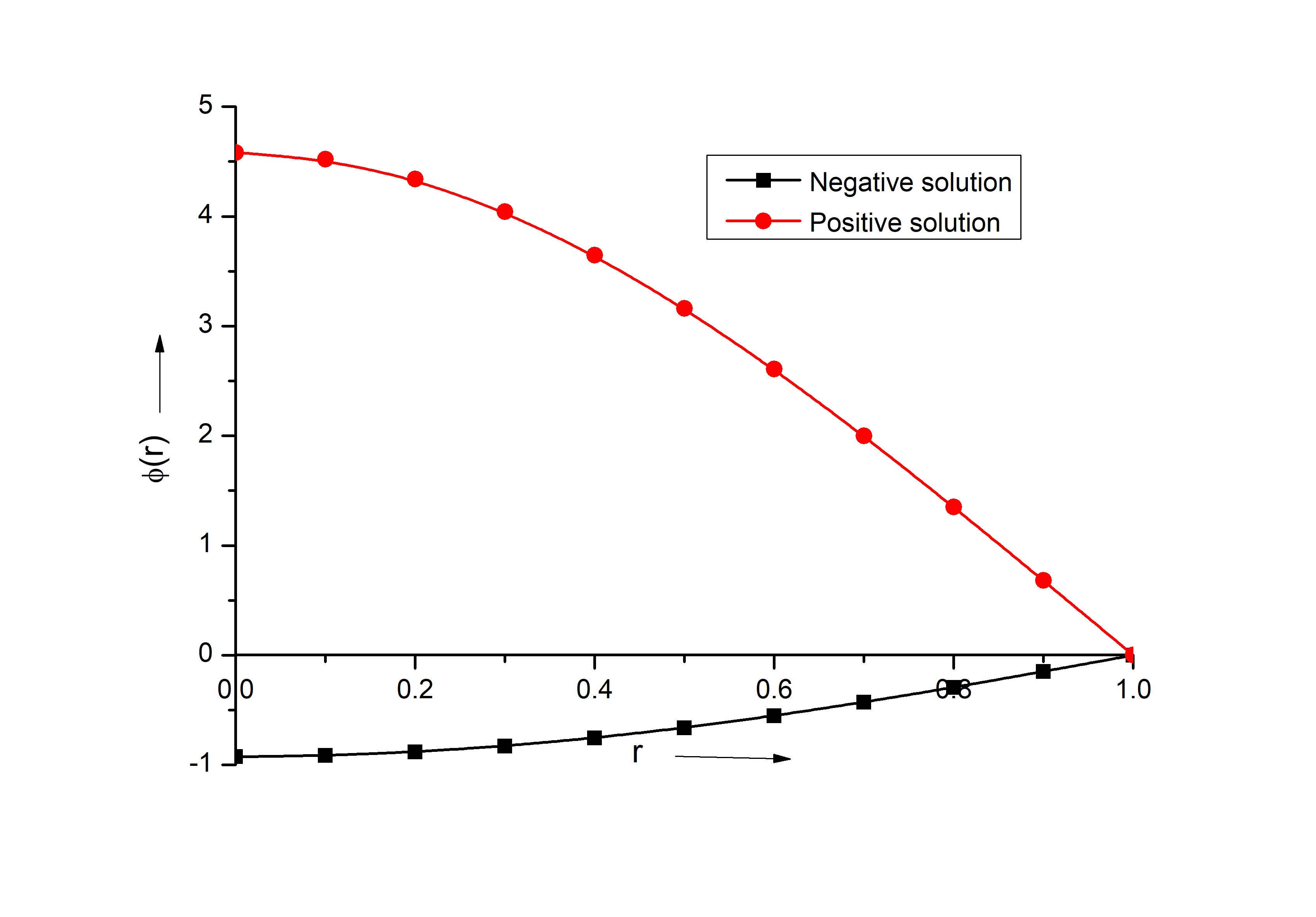}}
\caption{Approximate solutions $\phi(r)$ corresponding toequations (\ref{P5Intro10}) and (\ref{P5Eq205021}).}
\label{P5AG44}
\end{figure}
\begin{figure}[H]
\centering
\subfigure[\,\,$\lambda=0$]{\includegraphics[width=.5\linewidth]{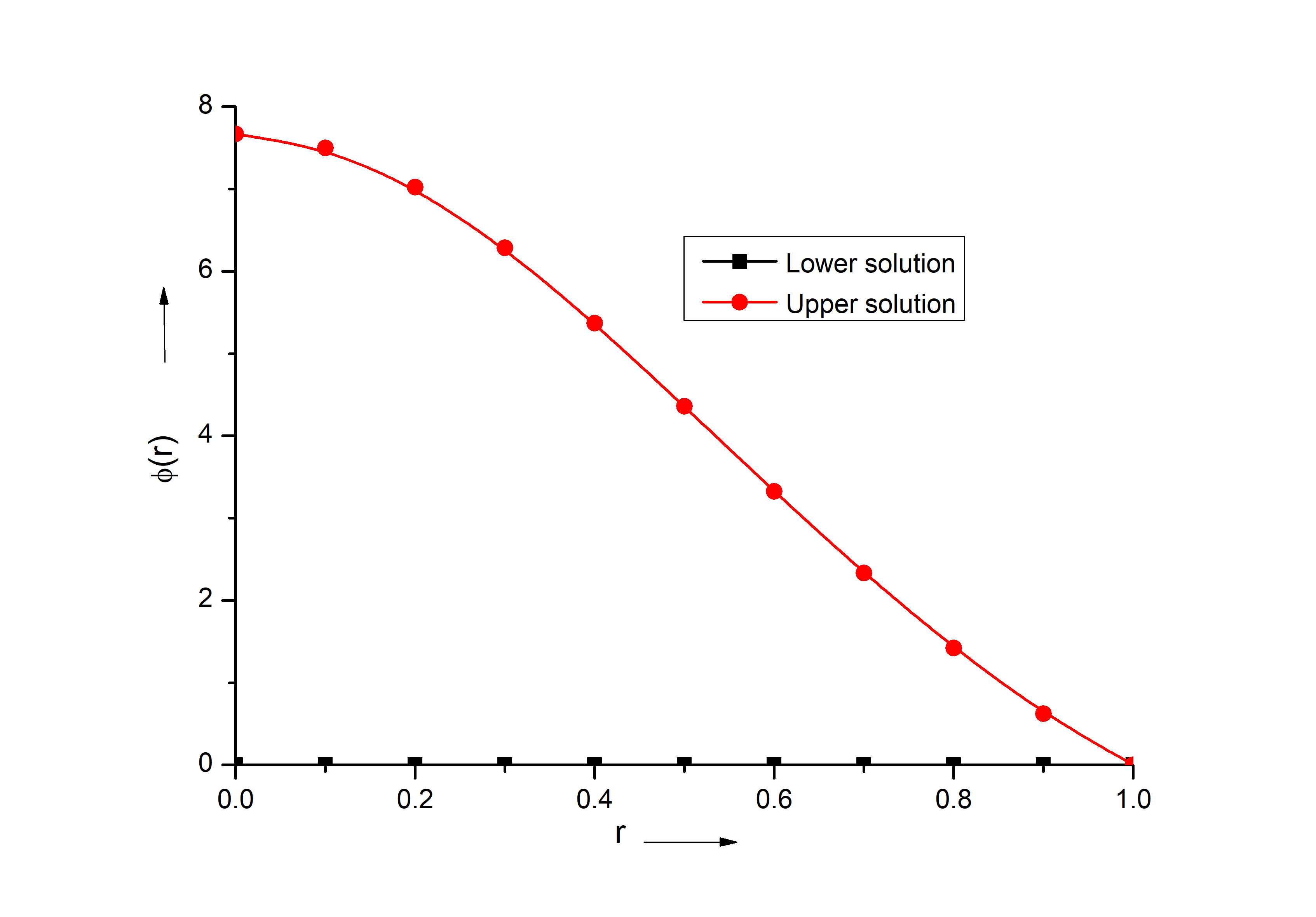}} \hspace{-.6cm}
\subfigure[\,\,$\lambda=168.76$]{\includegraphics[width=.5\linewidth]{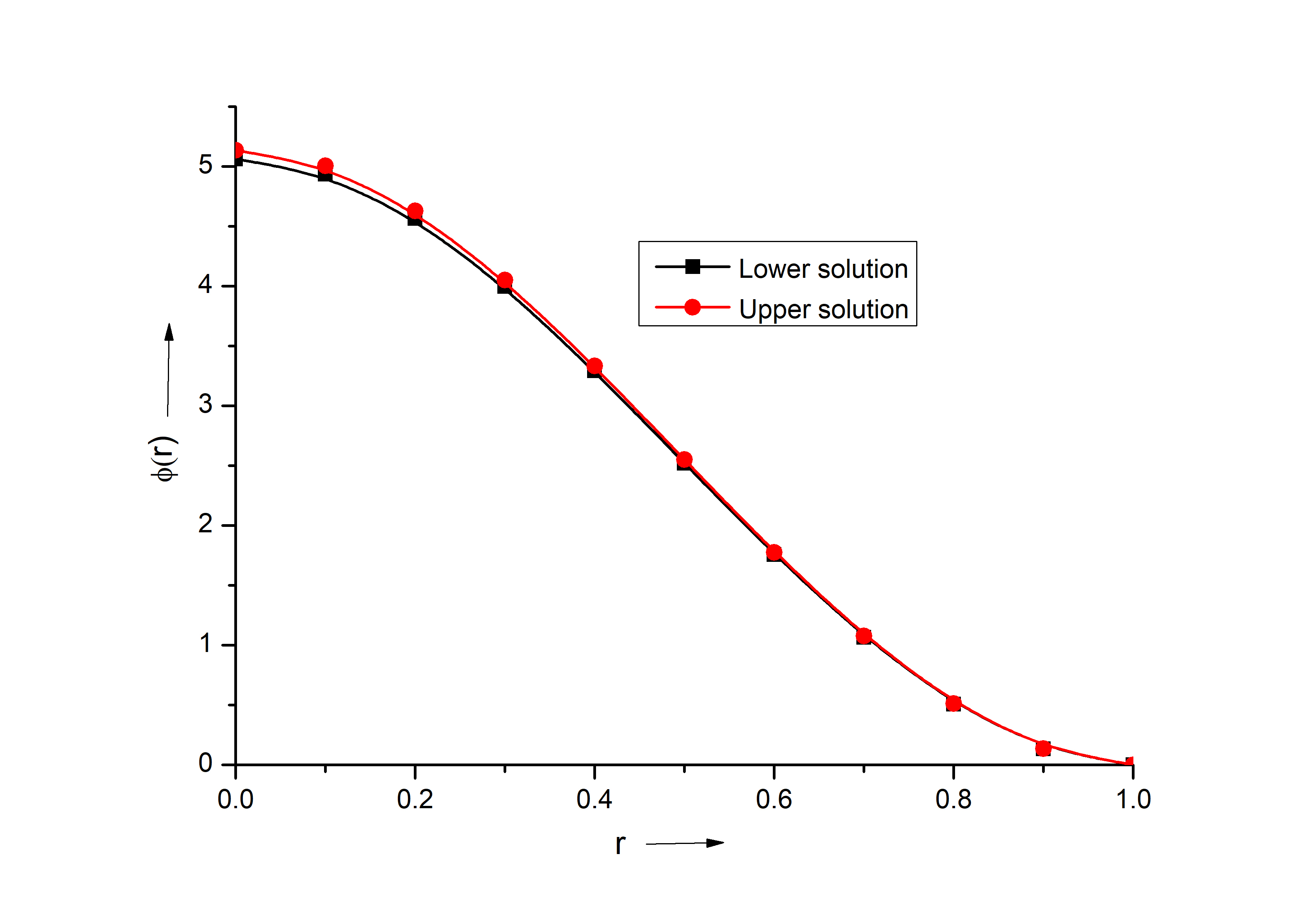}}
\caption{Approximate solutions $\phi(r)$ corresponding to equations (\ref{P5Intro10}) and (\ref{P5Eq205020}).}
\label{P5AG55}
\end{figure}
\begin{figure}[H]
\centering
\subfigure[\,\,$\lambda=-1$]{\includegraphics[width=.5\linewidth]{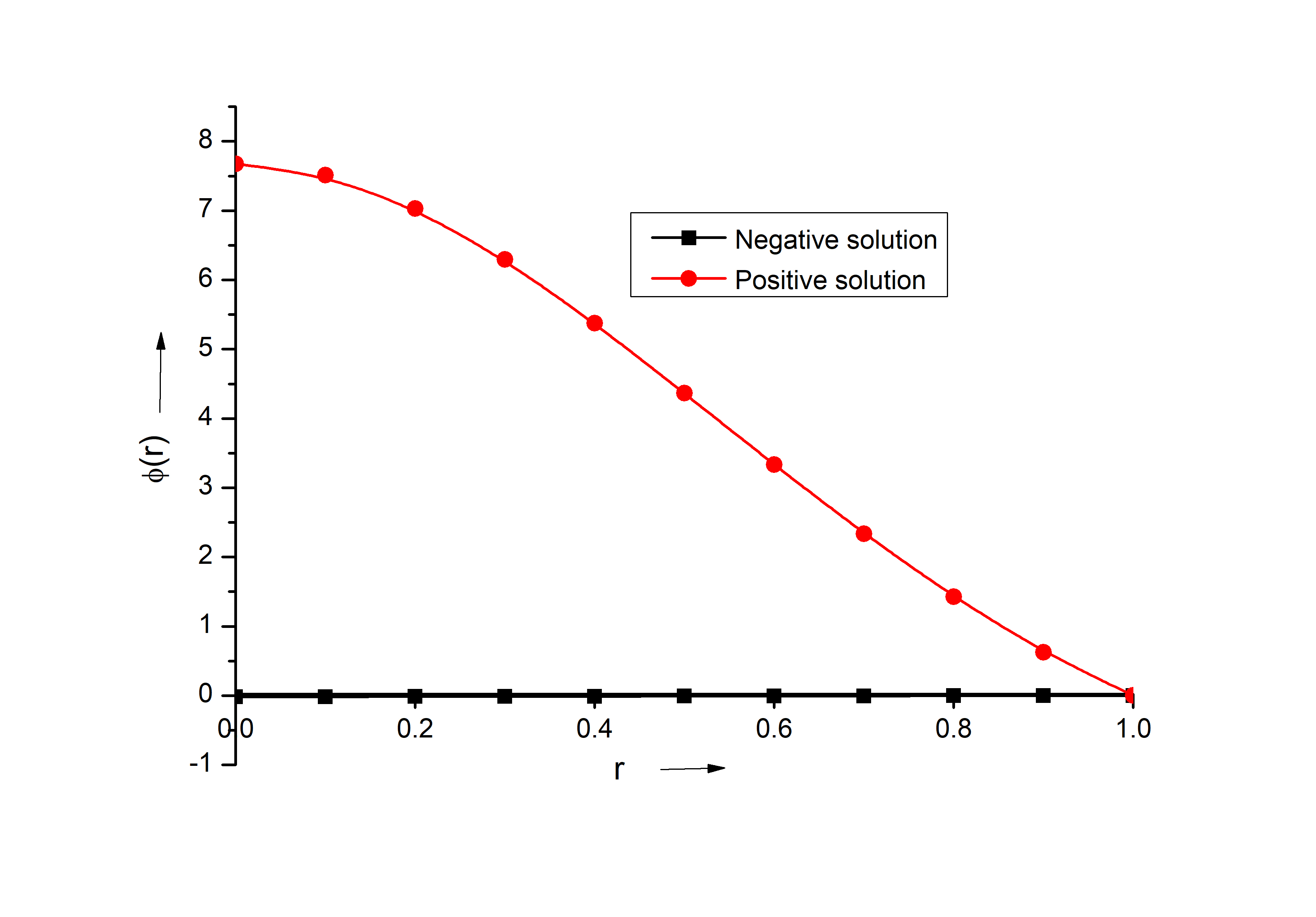}} \hspace{-.6cm}
\subfigure[\,\,$\lambda=-15$]{\includegraphics[width=.5\linewidth]{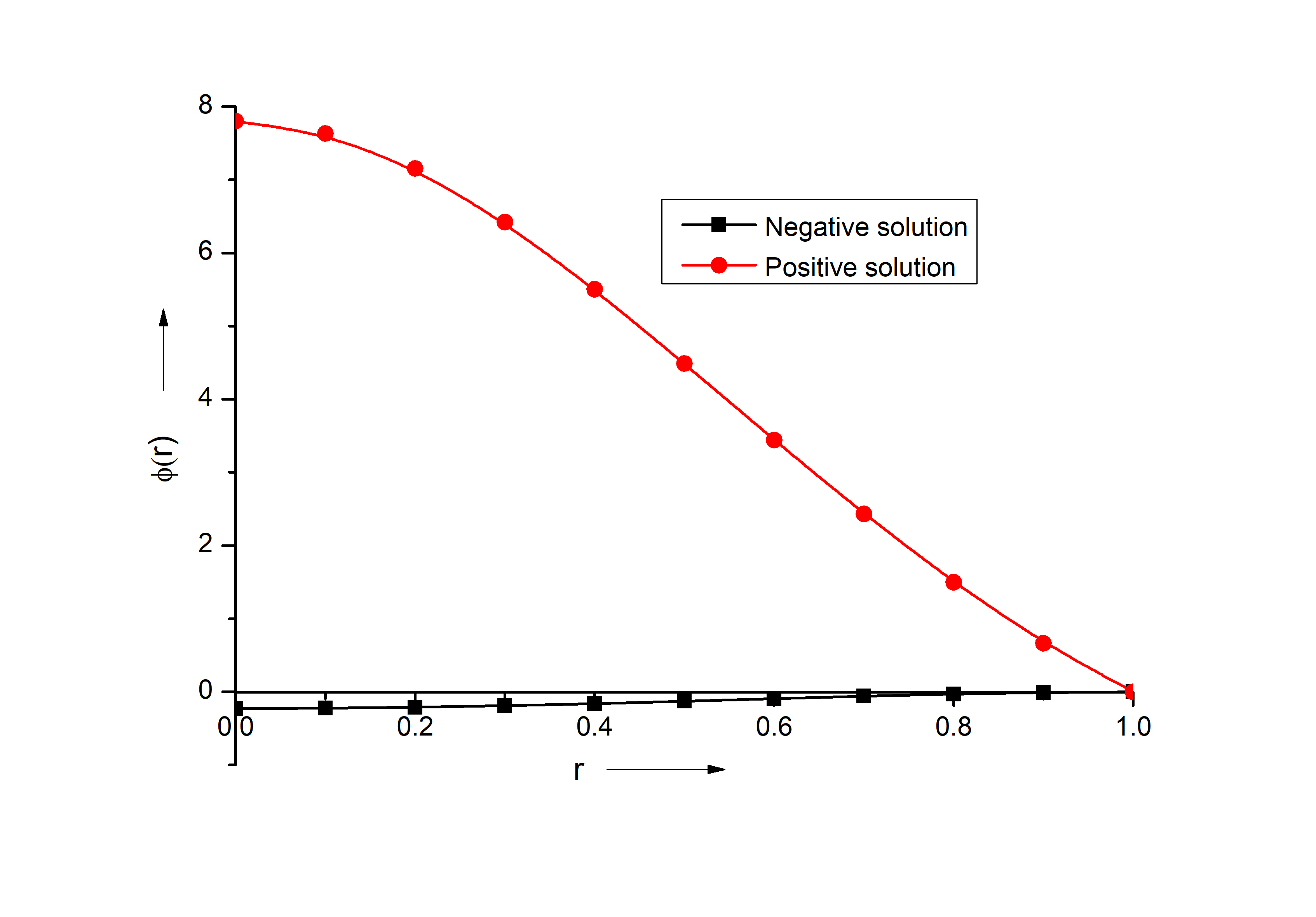}}
\caption{Approximate solutions $\phi(r)$ corresponding equations (\ref{P5Intro10}) and (\ref{P5Eq205020}).}
\label{P5AG66}
\end{figure}
\section{Conclusions}\label{P5Conclusions}
In this work, we derived some qualitative properties of the singular boundary value problems. Also, we  proved the existence of solution and find out a range of parameter $k$ for which the nonlinear problem has multiple solutions in the region $D_0$. We established the bounds of the parameter $\lambda$, from which we concluded about the nonexistence of solutions. All the results can make these problems very interesting and attracting for researchers. Also the boundary value problems have multiple solutions, therefore it is challenging for researchers to get an suitable scheme to capture both solutions with desired acuracy. But, here we  successfully developed  the iterative schemes,  and captured both solutions together with high acuracy. From tables \ref{P5table1}-  \ref{P5table4}, we saw that the approximate solutions computed by our proposed method  converge to the exact solutions very fast. But, corresponding to  boundary conditions (\ref{P5Eq205020}), we noticed that, positive approximate solution converge to exact positive solution very slowly ([See: table \ref{P5table6}]). We verified that our numerical results are well matched with our theoretical results as well as existing numerical results (\cite{Biswajit2019}). Among all point of view, we conclude that, our proposed technique is quit powerful and efficient. Furthermore, this technique will be an effective tool to solve BVPs, which have multiple solutions.

\section*{Acknowledgements}
This work is supported by grant provided by DST project, file name: SB/S4/MS/805/12 and INSPIRE Program Division, Department of Science $\&$ Technology, New Delhi, India-$110016$.
\bibliography{MasterR}
\end{document}